\pdfoutput=1  % force pdflatex on arXiv (must be within the first 5 lines)
\documentclass[letterpaper,10pt]{article}
\usepackage[margin=1in]{geometry}

\usepackage{graphicx}%
\usepackage{multirow}%
\usepackage{amsmath,amssymb,amsfonts}%
\usepackage{amsthm}%
\usepackage{mathrsfs}%
\usepackage{xcolor}%
\usepackage{textcomp}%
\usepackage{booktabs}%
\usepackage{float}%
\usepackage{algorithm}%
\usepackage{algpseudocode}%
\usepackage{subcaption}%
\usepackage{caption}
\usepackage{mathtools}
\usepackage{listings}%

\usepackage[authoryear,round]{natbib}
\setcitestyle{authoryear,round,citesep={;},aysep={,},yysep={;}}

\usepackage{hyperref}
\usepackage{url}

\usepackage{comment}
\usepackage{relsize}
\usepackage{ifthen}
\usepackage[colorinlistoftodos]{todonotes}

\usepackage{graphics} % for pdf, bitmapped graphics files
\usepackage{rotating}
\usepackage{color}
\usepackage{enumerate}
\usepackage[T1]{fontenc}
\usepackage{psfrag}
\usepackage{epsfig} % for postscript graphics files
\usepackage{booktabs}
\usepackage{graphicx,url}
\usepackage{multirow}
\usepackage{array}
\usepackage{latexsym}
\usepackage{amsfonts}
\usepackage{amsmath}
\usepackage{amssymb}
\usepackage{mathtools}
\usepackage{xstring}
\usepackage{multirow}
\usepackage{xcolor}
\usepackage{prettyref}
\usepackage{bigdelim}
\usepackage{listings}

\usepackage{enumitem}
\usepackage{xspace}
\usepackage{bm}
\usepackage{tikz}
\usetikzlibrary{matrix,calc}

\usepackage{amsmath,array,arydshln,xparse}
\usepackage{adjustbox}

\usepackage{mdwlist}

\makecompactlist{itemize}{stditemize}

\newrefformat{prob}{Problem\,\ref{#1}}
\newrefformat{def}{Definition\,\ref{#1}}
\newrefformat{sec}{Section\,\ref{#1}}
\newrefformat{sub}{Section\,\ref{#1}}
\newrefformat{prop}{Proposition\,\ref{#1}}
\newrefformat{app}{Appendix\,\ref{#1}}
\newrefformat{alg}{Algorithm\,\ref{#1}}
\newrefformat{cor}{Corollary\,\ref{#1}}
\newrefformat{thm}{Theorem\,\ref{#1}}
\newrefformat{lem}{Lemma\,\ref{#1}}
\newrefformat{fig}{Fig.\,\ref{#1}}
\newrefformat{tab}{Table\,\ref{#1}}

\newcommand{\bdmath}{\begin{dmath}}
\newcommand{\edmath}{\end{dmath}}
\newcommand{\beq}{\begin{equation}}
\newcommand{\eeq}{\end{equation}}
\newcommand{\bdm}{\begin{displaymath}}
\newcommand{\edm}{\end{displaymath}}
\newcommand{\bea}{\begin{eqnarray}}
\newcommand{\eea}{\end{eqnarray}}
\newcommand{\beal}{\beq \begin{array}{ll}}
\newcommand{\eeal}{\end{array} \eeq}
\newcommand{\beas}{\begin{eqnarray*}}
\newcommand{\eeas}{\end{eqnarray*}}
\newcommand{\ba}{\begin{array}}
\newcommand{\ea}{\end{array}}
\newcommand{\bit}{\begin{itemize}}
\newcommand{\eit}{\end{itemize}}
\newcommand{\ben}{\begin{enumerate}}
\newcommand{\een}{\end{enumerate}}

\newcommand{\myParagraph}[1]{{\bf #1.}\xspace}
\newcommand{\hide}[1]{}

\newcommand{\hiddenText}{{\color{gray} hidden text.}}
\newcommand{\hideWithText}[1]{\hiddenText}

\newcommand{\norm}[1]{\left\| #1 \right\|}

\newcommand{\blue}[1]{{\color{blue}#1}}

\newcommand{\linkToPdf}[1]{\href{#1}{\blue{(pdf)}}}
\newcommand{\linkToPpt}[1]{\href{#1}{\blue{(ppt)}}}
\newcommand{\linkToCode}[1]{\href{#1}{\blue{(code)}}}
\newcommand{\linkToWeb}[1]{\href{#1}{\blue{(web)}}}
\newcommand{\linkToVideo}[1]{\href{#1}{\blue{(video)}}}
\newcommand{\linkToMedia}[1]{\href{#1}{\blue{(media)}}}
\newcommand{\award}[1]{\xspace} % {{\red{#1}}} % omit awards

\newcommand{\Sn}{\mathbb{S}^n}
\newcommand{\R}{\mathbb{R}}

\renewcommand{\norm}[1]{\lVert #1 \rVert}
\newcommand{\inprod}[2]{\langle #1, #2 \rangle}

\newcommand{\sym}[1]{\mathbb{S}^{#1}}

\newcommand{\bmat}{\left[ \begin{array}}
\newcommand{\emat}{\end{array}\right]}
\newcommand{\psd}[1]{\sym{#1}_{+}}

\newcommand{\half}{\frac{1}{2}}

\newcommand{\abs}[1]{\left|#1\right|}

\newcommand{\ceil}[1]{\left\lceil #1 \right\rceil}

\hypersetup{colorlinks, hypertexnames=false, pageanchor=true,
            linkcolor=blue, citecolor={green!50!black}, urlcolor=cyan,
            pdftitle={Spectral Deflation for Factorization-Free Matrix Filtering in Muon and Semidefinite Programming}
            }
\mathtoolsset{centercolon}

\newtheorem{theorem}{Theorem}
\newtheorem{lemma}[theorem]{Lemma}
\newtheorem{proposition}[theorem]{Proposition}

\newtheorem{example}{Example}
\newtheorem*{theoremrestated}{Theorem~1 (Exact deflation raises the spectral floor; restated with the deflated-error bound)}
\newtheorem*{proprestated}{Proposition~2 (Deflated PSD projection; restated with the deflated-error bound)}
\graphicspath{{fig/}}

\title{Spectral Deflation for Factorization-Free Matrix Filtering\\ in Muon and Semidefinite Programming}

\author{%
  Haoran Sun\thanks{School of Engineering and Applied Sciences, Harvard University.}\\
  \texttt{elvis271828@gmail.com}
  \and
  Shucheng Kang\footnotemark[2]\\
  \texttt{skang1@g.harvard.edu}
  \and
  Heng Yang\footnotemark[2]\\
  \texttt{hankyang@g.harvard.edu}
}

\begin{document}

\maketitle

%!TEX root = ../main.tex

\begin{abstract}
GPU implementations of the Muon optimizer and of first-order semidefinite programming (SDP) solvers share one computational pattern: a matrix factorization is replaced by a fixed-depth polynomial filter applied after normalization.
When a few dominant spectral components carry most of the input's scale, normalization pushes the remaining spectrum toward zero, where the filter is least accurate.
We propose \emph{spectral deflation}: estimate the dominant components, remove them, filter the normalized residual with the unchanged filter, and restore them.
Deflation preserves the target matrix function, and we prove that it strictly reduces the finite-step error of the classical Newton--Schulz family.
Implemented with batched randomized SVD for Muon and warm-started subspace tracking for ADMM, deflation consistently improves GPT-2 pretraining over the corresponding Muon baselines with both the Newton--Schulz and Polar Express mappings, also in wall-clock time, and lowers the KKT residuals of factorization-free ADMM on large-scale SDPs within a similar projection time.
\end{abstract}

%!TEX root = ../main.tex

\section{Introduction}
\label{sec:intro}

Large-scale numerical computing on GPUs increasingly replaces matrix factorizations with matrix--matrix products.
A prominent instance of this trend approximates a spectral matrix function $f(A)$ by a \emph{fixed-depth polynomial filter}: the input is normalized once, and a short fixed sequence of polynomials is applied,
\begin{equation}
X_0=\frac{A}{\theta(A)},
\qquad
X_t=p_t(X_{t-1}),\quad t=1,\ldots,L,
\qquad
X_L\approx f(A),
\label{eq:filter}
\end{equation}
where $\theta(A)>0$ is a normalization scale and each $p_t$ is evaluated by a few matrix--matrix multiplications.
The filter avoids any factorization and is therefore fast and low-precision friendly; however, because the depth $L$ is fixed, its accuracy depends strongly on the normalized spectrum of $X_0$.

\myParagraph{Two applications}
In deep learning, the Muon optimizer updates every matrix-valued parameter $W$ in the direction of the polar factor of its momentum matrix $M\in\R^{n\times m}$~\citep{jordan2024muon,liu2025muonscalable}:
$W\leftarrow W-\lambda\operatorname{polar}(M)$, where $\operatorname{polar}(M):=UV^{\top}$ for a compact singular value decomposition $M=U\Sigma V^{\top}$.
In words, Muon keeps the singular directions of $M$ and maps every nonzero singular value to one.
An SVD at every optimizer step is too expensive, so Muon computes \eqref{eq:filter} with $\theta=\norm{M}_F$ and five fixed Newton--Schulz-type polynomials.
In numerical optimization, first-order semidefinite programming (SDP) solvers repeatedly project a symmetric matrix $Z$ onto the positive semidefinite (PSD) cone~\citep{wen2010admm,odonoghue2016scs,boyd2011admm}.
The projection applies the ReLU function $\max(x,0)$ to the eigenvalues of $Z$.
Since $\max(x,0)=x\bigl(1+\operatorname{sign}(x)\bigr)/2$, computing the projection reduces to computing the matrix sign function $\operatorname{sign}(Z)$.
Recent factorization-free solvers again approximate $\operatorname{sign}(Z)$ by \eqref{eq:filter}, where $\theta$ is an estimated upper bound on $\norm{Z}_2$ and the fixed polynomials $p_t$ form a composite sign filter~\citep{kang2025psdprojection}.
Muon and factorization-free ADMM therefore share one computational core, and we treat them together.

\myParagraph{The spectral floor}
In both applications, the matrices entering \eqref{eq:filter} have spectra of a specific shape.
Momentum matrices in GPT-2 training carry a few large singular values and many near zero (Figure~\ref{fig:spectrum}; see also \citealp{magakyan2026spectral}), and matrices inside SDP solvers often approach a head-dominated regime~\citep{kang2025psdprojection}.
Normalizing by the whole matrix then compresses the tail of the spectrum.
We call the smallest normalized nonzero spectral magnitude of $X_0$ its \emph{spectral floor}.
To see the problem, assume that Muon's polynomial has slope $a>1$ at the origin, so five iterations amplify a small singular value by at most $a^5$; a singular value below $1/a^5$ of the Frobenius norm therefore ends far from one, and its direction is strongly attenuated in the computed polar factor.
This raises the question at the center of this paper:
\begin{center}
\emph{Can we lift the small spectral components that a fixed-depth polynomial filter fails to resolve,
without redesigning the filter and without any matrix factorization?}
\end{center}
We give an affirmative answer through a simple principle: \emph{deflate} the dominant components, \emph{filter} the normalized residual, and \emph{restore} what was removed.

\myParagraph{Spectral deflation}
Let $M=U_k\Sigma_kV_k^{\top}+R_k$ split off the $k$ largest singular triplets of $M$, and let $Z=U_k\Lambda_kU_k^{\top}+R_k$ analogously split off the $k$ eigenvalues of largest magnitude of a symmetric $Z$.
Both targets decompose exactly:
\begin{equation}
\operatorname{polar}(M)=U_kV_k^{\top}+\operatorname{polar}(R_k),
\qquad
\Pi_{\psd{n}}(Z)=U_k(\Lambda_k)_{+}U_k^{\top}+\Pi_{\psd{n}}(R_k),
\label{eq:identities}
\end{equation}
where $(\Lambda_k)_{+}$ zeroes the negative eigenvalues of the head.
In words, deflation never changes the target; it only changes the input that the filter sees.
Applying \eqref{eq:filter} to $R_k$ instead of the full matrix raises every remaining spectral magnitude by the factor $\alpha_k=\theta(A)/\theta(R_k)>1$, which is large exactly when the head dominates, and the deflated head is handled explicitly through \eqref{eq:identities}.
The filter itself is untouched.
This distinguishes our approach from work that redesigns the polynomials~\citep{grishina2025cans,amsel2025polar}, from Muon variants that change the update's scaling or geometry (Appendix~\ref{sec:related_work}), and from randomized low-rank polar approximations that replace the full-rank target~\citep{choudhury2026muon}. Classical deflation treats obstructive eigencomponents separately in Krylov methods~\citep{saad2000deflated,eiermann2011deflated}, and we adapt the same idea to dense fixed-depth filters, where the obstacle is the normalization scale.
Figure~\ref{fig:intro_loss} previews the end-to-end effect on both applications.

\begin{figure}[t]
\centering
% The two boxes are sized so that both panels come out the same height:
% the loss figure is 873 x 228 pt (aspect 3.82), the two-instance ADMM panel
% 338 x 174 pt (aspect 1.94), so 0.657/3.82 = 0.333/1.94 (up to rounding);
% the boxes sum to 0.99\linewidth so the gap between (a) and (b) is small.
% The ADMM panel is produced by defadmm/scripts/plot_intro_panel.py G55mc,G59mc.
\begin{subfigure}[b]{0.657\linewidth}
\centering
\includegraphics[width=\linewidth]{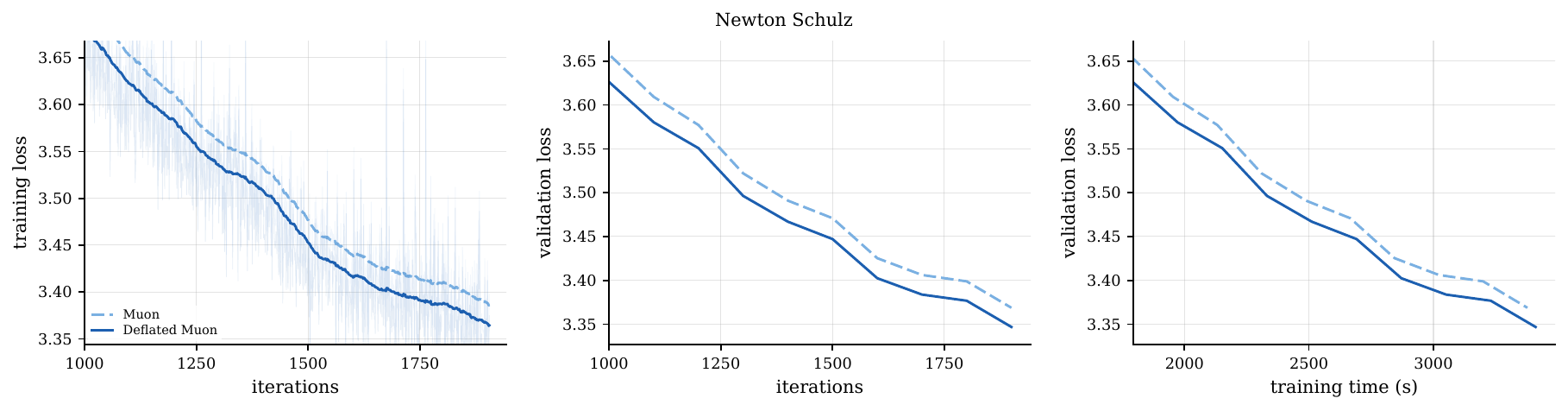}
\caption{\textsc{Deflated Muon}}
\label{fig:intro_muon}
\end{subfigure}\hfill
\begin{subfigure}[b]{0.333\linewidth}
\centering
\includegraphics[width=\linewidth]{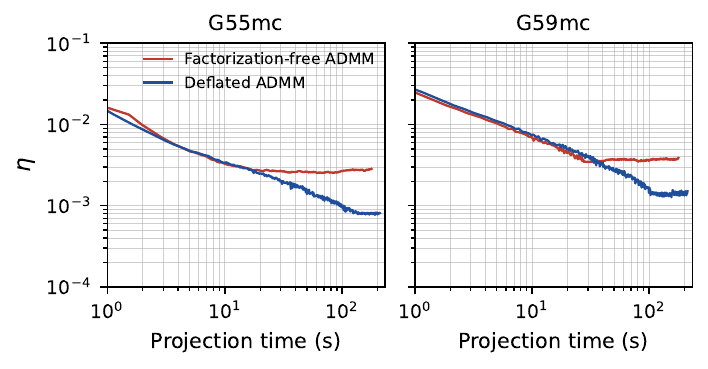}
\caption{ADMM with deflated projection}
\label{fig:intro_admm}
\end{subfigure}
\caption{(a) Training loss against optimizer iterations (left), validation loss against optimizer iterations (middle) and validation loss against training time (right) of GPT-2 Large trained on one billion tokens from FineWeb;
we compare Muon with the Newton--Schulz mapping against \textsc{Deflated Muon}.
(b) KKT residual $\eta$ of ADMM on the max-cut SDPs \texttt{G55mc} and \texttt{G59mc} ($n=5{,}000$) against cumulative projection time, with every PSD projection performed in FP16:
the factorization-free filter of \citet{kang2025psdprojection} (red) and the same filter behind deflation (blue).}
\label{fig:intro_loss}
\end{figure}

\myParagraph{Contributions}
Our contributions are threefold.
(i) We identify the spectral floor as a shared accuracy bottleneck of fixed-depth polynomial filters in Muon and in factorization-free ADMM, and we formalize \emph{spectral deflation} as a target-preserving wrapper around an unchanged filter, based on the exact decompositions \eqref{eq:identities}.
(ii) We instantiate the wrapper as two GPU primitives: a \emph{deflated polar factor} for Muon, driven by a batched randomized SVD that processes all momentum matrices of one shape at once (\S\ref{sec:deflated_muon}), and a \emph{deflated PSD projection} for ADMM, driven by a warm-started subspace iteration that tracks the dominant eigenspace across consecutive projections (\S\ref{sec:deflated_psd}).
On the theory side, we prove that exact deflation raises the spectral floor from $\ell$ to $\min(\alpha_k\ell,\gamma^{-1})$, where $\gamma$ is the padding parameter, and contracts the finite-step Newton--Schulz error accordingly (Theorem~\ref{thm:deflated_ns}), and we prove the analogous statement for the Newton--Schulz PSD projection (Proposition~\ref{prop:deflated_psd_ns}); Appendix~\ref{app:deflated_admm_error} gives a matching bound for the fixed projection filter used in practice.
(iii) Empirically, deflation improves both applications under the same filter budget: pretraining GPT-2 Large on FineWeb, \textsc{Deflated Muon} attains lower validation loss than Muon across learning rates, momentum coefficients, iteration counts, and training horizons, for both mappings and also against wall-clock time; and on four large SDPs with FP16 projections, ADMM with the deflated projection reaches lower KKT residuals within the same projection-time budget.

\myParagraph{Organization}
\S\ref{sec:deflation} develops spectral deflation: the mechanism (\S\ref{sec:spectrum}), the two primitives (\S\ref{sec:deflated_muon}--\ref{sec:deflated_psd}), and the error analysis (\S\ref{sec:error_analysis}).
\S\ref{sec:experiments} evaluates matrix-function accuracy, GPT-2 pretraining, and SDP solving; \S\ref{sec:conclusion} concludes.
Related work is reviewed in Appendix~\ref{sec:related_work}.

% \myParagraph{Notation}
% $\sigma_1(A)\geq\sigma_2(A)\geq\cdots$ are singular values and $\lambda_i(Z)$ eigenvalues of a symmetric $Z$; $\norm{\cdot}_F$ and $\norm{\cdot}_2$ are the Frobenius and spectral norms; $\Sn$ and $\psd{n}$ are the symmetric matrices and the PSD cone, with $\Pi_{\psd{n}}$ the Euclidean projection onto $\psd{n}$; $(x)_+:=\max\{x,0\}$; and $\sigma(X)$ denotes the set of nonzero singular values of $X$.

%!TEX root = ../main.tex

\section{Spectral Deflation}
\label{sec:deflation}

\myParagraph{Notation}
$\sigma_1(A)\geq\sigma_2(A)\geq\cdots$ are singular values and $\abs{\lambda_1(Z)}\geq\abs{\lambda_2(Z)}\geq\cdots$ eigenvalues of a symmetric $Z$ ordered by magnitude; $\norm{\cdot}_F$ and $\norm{\cdot}_2$ are the Frobenius and spectral norms; $\Sn$ and $\psd{n}$ are the symmetric matrices and the PSD cone, with $\Pi_{\psd{n}}$ the Euclidean projection onto $\psd{n}$; $(x)_+:=\max\{x,0\}$; and $\sigma(X)$ denotes the set of nonzero singular values of $X$.

\subsection{The Spectral Floor}
\label{sec:spectrum}

We first make the mechanism precise for Muon.
For the polar factor, the filter \eqref{eq:filter} with $\theta=\norm{M}_F$ is a Newton--Schulz iteration;
we write it using the classical degree-five map~\citep{bjorck1971iterative},
\begin{equation}
X_0
=
\frac{M}{\norm{M}_F},
\
X_{t+1}
=
\mathcal{N}(X_t)
:=
aX_t
+b(X_tX_t^{\top})X_t
+c(X_tX_t^{\top})^2X_t,
\ t=0,\ldots,L-1,
\label{eq:muon_ns}
\end{equation}
with $a=15/8$, $b=-5/4$, and $c=3/8$ (Muon's implementation retunes these coefficients for a larger slope at the origin; \S\ref{subsec:convergence_experiments} gives the values).
At the spectral level, each step applies the scalar polynomial $p(x)=ax+bx^3+cx^5$ to every singular value; under this map, every nonzero singular value of $X_0$ increases monotonically to one, so $X_t\to\operatorname{polar}(M)$ as $t\to\infty$.
The difficulty lies in the fixed budget $L$: since $p(x)\leq ax$ on $[0,1]$, a singular value starting at $x$ reaches at most $a^Lx$ after $L$ steps, so any $x\ll a^{-L}$ remains far from one within the budget.

\myParagraph{What deflation does}
For a momentum matrix $M$, the singular values of $X_0$ in \eqref{eq:muon_ns} are the Frobenius-normalized values $\hat{\sigma}_i(M)=\sigma_i(M)/\norm{M}_F$.
Across the attention and MLP momentum matrices recorded during GPT-2 training (setup in \S\ref{subsec:convergence_experiments}), these spectra share a common structure: a few large singular values---the \emph{head}---dominate, while the remaining ones---the \emph{tail}---concentrate near zero (Figure~\ref{fig:spectrum} in Appendix~\ref{app:toy_deflation}).
The spectral floor of $X_0$ is then tiny, and the fixed-depth filter fails to resolve the tail, as discussed in \S\ref{sec:intro}.
Deflation changes the normalization, not the polynomial: consider the exact rank-$k$ decomposition $M=U_k\Sigma_kV_k^{\top}+R_k$ with $1\le k<q=\operatorname{rank}(M)$, and normalize the residual instead of $M$.
Every remaining singular value is then rescaled by
\begin{equation}
\frac{\sigma_i(M)}{\norm{M}_F}
\;\longrightarrow\;
\frac{\sigma_i(M)}{\norm{R_k}_F}
=\alpha_k\,\frac{\sigma_i(M)}{\norm{M}_F},
\ ~\text{where}~
\alpha_k:=\frac{\norm{M}_F}{\norm{R_k}_F}>1,
\label{eq:tail_amplification}
\end{equation}
while the target is unchanged by \eqref{eq:identities}: the deflated head is restored explicitly, and only the residual passes through the filter.
When the head carries most of the Frobenius mass, $\alpha_k$ is large, and the whole tail is lifted into a region where the fixed-depth approximation is substantially more accurate.
\S\ref{sec:error_analysis} makes this quantitative: the spectral floor of the initialization rises by exactly the factor $\alpha_k$, and this factor controls both error bounds of Theorem~\ref{thm:deflated_ns}.

\begin{example}[A head-dominated toy spectrum]
\label{ex:toy}
Let $M$ have singular values $(10,1,0.5,0.2,0.1)$, so the single head $\sigma_1$ dominates the Frobenius norm, $\norm{M}_F\approx10.06$.
Under the standard normalization, the smallest singular value starts at the spectral floor $\sigma_5/\norm{M}_F\approx10^{-2}$; after five steps of \eqref{eq:muon_ns} it has reached only $0.23$.
Deflating $\sigma_1$ gives $\norm{R_1}_F\approx1.14$, hence $\alpha_1\approx8.8$: the same singular value now starts at $8.8\times10^{-2}$ and reaches $0.988$ after five steps, and the remaining tail values, lifted by the same factor, converge earlier still (Figure~\ref{fig:toy_deflation} in Appendix~\ref{app:toy_deflation}).
\end{example}

One computational question remains: deflation needs the leading spectral components of every input matrix at every step, and estimating them must cost far less than the filter itself.
The next two subsections give one such estimator for each application.

\subsection{Deflated Polar Factor and \textsc{Deflated Muon}}
\label{sec:deflated_muon}

Suppose estimated leading triplets $(U,s,V)$ of $M$ are available, with $s_1\geq\cdots\geq s_r$ (from Algorithm~\ref{alg:brsvd} below); Algorithm~\ref{alg:deflated_muon} states the resulting primitive.
Deflation activates only when the gate $s_r<\tau s_1$ fires, which implies $k<r$: the computed window always contains the entire deflated head, and the same $\tau$ serves as gate and rank threshold.

\begin{algorithm}[t]
\caption{\textsc{Deflated Polar Factor} (one matrix)}
\label{alg:deflated_muon}
\begin{algorithmic}[1]
\Require $M\in\R^{n\times m}$, $n\ge m$; its leading triplets $(U,s,V)$ from Algorithm~\ref{alg:brsvd}; threshold $\tau\in(0,1)$; padding $\gamma>1$; polynomial map $\mathcal{N}$; iteration count $L\ge0$
\If{$s_r/s_1<\tau$} \Comment{deflation gate}
   \State $k\gets\max\{i:s_i>\tau s_1\}$;\quad $U_k,V_k\gets$ leading $k$ columns of $U,V$ \Comment{the gate implies $k<r$}
   \State $R\gets M-U_k\,\operatorname{diag}(s_1,\ldots,s_k)\,V_k^{\top}$
   \State $X_0\gets R/\norm{R}_F+\gamma^{-1}U_kV_k^{\top}$
\Else
   \State $X_0\gets M/\norm{M}_F$
\EndIf
\State \Return $\mathcal{N}^{\circ L}(X_0)$
\end{algorithmic}
\end{algorithm}

\myParagraph{Padding}
Because the estimated triplets are inexact, the subtraction leaves small components along the head directions, which the polynomial's large slope near the origin would amplify; we therefore reinsert the estimated head into the filter initialization at magnitude $\gamma^{-1}$.
The margin $\gamma>1$ is a practical safeguard against overshoot from approximate deflation and low-precision arithmetic (with exact triplets, $\norm{X_0}_2\leq1$ holds automatically).

\myParagraph{Estimating the head}
Deflation needs only the leading $r\ll m$ singular triplets, which randomized SVD estimates from a sketch of $r'=r+\ceil{\delta m}$ columns, where $\delta$ is an oversampling ratio~\citep{halko2011finding}.
Applying it to each momentum matrix separately, however, is inefficient on a GPU for two reasons.
First, momentum matrices are small relative to the device, so a per-matrix loop is dominated by kernel-launch latency rather than arithmetic.
Second, all matrices of one shape can share every stage of the computation.
We therefore stack same-shape matrices into a tensor $A\in\R^{B\times n\times m}$ and run one \textsc{Batched Randomized SVD} per group (Algorithm~\ref{alg:brsvd}, Appendix~\ref{app:brsvd}): a Gaussian sketch, CholeskyQR orthonormalization, $J$ batched subspace iterations, a thin SVD of the reduced $r'\times m$ matrices $Q^{\top}A$, and truncation to the leading $r=\ceil{wm}$ triplets, where the window ratio $w\in(0,1)$ sets what fraction of the spectrum is estimated; wide matrices are handled through $\operatorname{polar}(M^{\top})=\operatorname{polar}(M)^{\top}$.
Every large operation is a batched GEMM or a batched CholeskyQR, and an SVD touches only the reduced matrices, so the front-end is asymptotically cheaper than the polynomial iterations.

\myParagraph{Muon with deflation}
\textsc{Deflated Muon} replaces the polar-factor routine inside Muon by Algorithm~\ref{alg:deflated_muon}: at every optimizer step, Algorithm~\ref{alg:brsvd} runs once per shape group, and each momentum matrix is then processed for $L=5$ steps with the practical mapping in use---Muon's retuned map or Polar Express (\S\ref{subsec:convergence_experiments}).
The primitive is agnostic to $\mathcal{N}$; the classical map \eqref{eq:muon_ns} is the reference for the analysis in \S\ref{sec:error_analysis}.
When the gate does not fire, $X_0=M/\norm{M}_F$ and the method reduces exactly to Muon.
The values of $w$, $\tau$, and $\gamma$ are reported in \S\ref{sec:experiments}.

\subsection{Deflated PSD Projection and ADMM}
\label{sec:deflated_psd}

We now instantiate the same wrapper for PSD projection.
For a symmetric matrix $Z=V\Lambda V^{\top}$, the PSD-cone projection is a matrix ReLU~\citep{higham1988nearest}, 
and it can be rewritten through the matrix sign function:
\begin{equation}
\Pi_{\psd{n}}(Z)
=
V\max(\Lambda,0)V^{\top}
=
\half Z\bigl(I+\operatorname{sign}(Z)\bigr).
\label{eq:psd_sign}
\end{equation}
Factorization-free solvers approximate $\operatorname{sign}(Z)$ by a fixed composite odd filter $g=p_T\circ\cdots\circ p_1$ built offline (Appendix~\ref{app:filter}; \citealp{kang2025psdprojection,lee2022minimax}), which yields the baseline projection
\begin{equation}
\widehat{\Pi}_{+}(Z)
=
\half Z\bigl(I+g(Z/\theta)\bigr),
\label{eq:baseline_projection}
\end{equation}
where $\theta$ is a Lanczos-based upper estimate of $\norm{Z}_2$ that places the spectrum of $Z/\theta$ in $[-1,1]$ (Appendix~\ref{app:admm_sdp}).
The deflated projection, stated as Algorithm~\ref{alg:deflation_admm}, mirrors Algorithm~\ref{alg:deflated_muon}.
Its gate acts on the magnitudes $\abs{\lambda_i}$ of the estimated eigenpairs $(\lambda_i,u_i)$, because the sign filter must resolve both large positive and large negative eigenvalues; the deflated directions are padded at $\pm\gamma^{-1}$; and the residual $R=Z-\sum_{i\in\mathcal{I}}\lambda_iu_iu_i^{\top}$, normalized by its own Lanczos-based upper estimate $\hat{\theta}\geq\norm{R}_2$, passes through the same filter $g$.
When reconstructing the projection, only the positive deflated eigencomponents are restored, since $(\lambda)_+=0$ for $\lambda<0$.
To ensure that $P$ is symmetric despite rounding errors in the reconstruction, we symmetrize it explicitly, $P\leftarrow\half(P+P^{\top})$.

\myParagraph{Warm-started subspace tracking}
The estimator differs from Muon's for a structural reason.
Momentum matrices are numerous and grouped by shape, so the polar routine sketches each group afresh at every optimizer step.
ADMM instead produces one large, slowly varying symmetric matrix per iteration, so its dominant eigenspace can be carried across iterations.
We maintain an orthonormal basis $V\in\R^{n\times K}$ with $K=\ceil{w'n}$: each projection performs one block step $\widetilde{Z}=ZV$ and one Rayleigh--Ritz step, uses the leading $r=\ceil{wn}\leq K$ Ritz magnitudes for the gate, and orthonormalizes $\widetilde{Z}$ as the next basis, in the spirit of warm-started block eigensolvers~\citep{knyazev2001lobpcg}.
Because consecutive iterates are close, one block subspace step per iteration is sufficient in practice; $V$ starts as a Gaussian random block, and the gate stays off during an initial warm-up phase while the basis stabilizes.
Algorithm~\ref{alg:deflation_admm} summarizes one projection; the constants are reported in \S\ref{sec:sdp_experiments}.

\begin{algorithm}[t]
\caption{\textsc{Deflated PSD Projection} (one projection)}
\label{alg:deflation_admm}
\begin{algorithmic}[1]
\Require $Z\in\Sn$; basis $V\in\R^{n\times K}$, $K=\ceil{w'n}$, from the previous iteration; window ratios $w\leq w'$; threshold $\tau$; padding $\gamma$; composite filter $g=p_T\circ\cdots\circ p_1$
\State $\widetilde{Z}\gets ZV$;\quad
eigendecompose $V^{\top}\widetilde{Z}=E\operatorname{diag}(\lambda)E^{\top}$;\quad
$U\gets VE$ \Comment{Ritz pairs $(\lambda_i,u_i)$}
\State $s\gets$ the values $\abs{\lambda_i}$ sorted decreasingly;\quad
$r\gets\ceil{wn}$
\If{$s_r<\tau s_1$} \Comment{the gate of Algorithm~\ref{alg:deflated_muon}, on magnitudes}
   \State $\mathcal{I}\gets\{i:\abs{\lambda_i}>\tau s_1\}$
   \State $R\gets Z-\sum_{i\in\mathcal{I}}\lambda_iu_iu_i^{\top}$
   \State $\hat{\theta}\gets$ Lanczos-based upper estimate of $\norm{R}_2$
   \State $X_0\gets R/\hat{\theta}+\gamma^{-1}\textstyle\sum_{i\in\mathcal{I}}\operatorname{sign}(\lambda_i)u_iu_i^{\top}$ \Comment{padded initialization, as in Algorithm~\ref{alg:deflated_muon}}
   \State $P\gets
   \half R\bigl(I+g(X_0)\bigr)
   +\textstyle\sum_{i\in\mathcal{I},\,\lambda_i>0}
   \lambda_iu_iu_i^{\top}$ \Comment{deflated projection}
\Else
   \State $\theta\gets$ Lanczos-based upper estimate of $\norm{Z}_2$;\quad
   $X_0\gets Z/\theta$
   \State $P\gets
   \half Z\bigl(I+g(X_0)\bigr)$ \Comment{eq.~\eqref{eq:baseline_projection}}
\EndIf
\State $V\gets\mathrm{CholeskyQR}(\widetilde{Z})$
\State \Return $P$ and the updated basis $V$
\end{algorithmic}
\end{algorithm}

\myParagraph{Use inside ADMM}
We evaluate the deflated projection inside a first-order solver for the standard-form semidefinite program and its dual,
\begin{equation}
\begin{array}{@{}rl@{\qquad\qquad}rl@{}}
\text{Primal:} & \displaystyle\min_{X}\ \inprod{C}{X}
& \text{Dual:} & \displaystyle\max_{y,S}\ b^{\top}y \\[4pt]
& \text{s.t.}\ \mathcal{A}X=b,\ X\in\psd{n},
& & \text{s.t.}\ \mathcal{A}^{*}y+S=C,\ S\in\psd{n},
\end{array}
\label{eq:sdp_pair}
\end{equation}
where $\mathcal{A}:\Sn\to\R^m$ with $\mathcal{A}X=(\inprod{A_1}{X},\ldots,\inprod{A_m}{X})$ and $\mathcal{A}^{*}y=\sum_i y_iA_i$.
A classical three-block ADMM~\citep{wen2010admm} alternates a linear solve for $y$, the projection $S^{(k+1)}=\Pi_{\psd{n}}(Z^{(k+1)})$ with $Z^{(k+1)}:=C-\mathcal{A}^{*}y^{(k+1)}-\sigma^{-1}X^{(k)}$, and a matrix addition for $X$ (the full iteration is Appendix~\ref{app:admm_sdp}, eq.~\eqref{eq:admm_iter}).
For large instances the projection dominates the per-iteration cost~\citep{kang2025psdprojection,groudiev2025cuadmm}; we leave the $y$- and $X$-updates untouched and replace only the projection by Algorithm~\ref{alg:deflation_admm}.

\subsection{Error Analysis}
\label{sec:error_analysis}
\label{subsec:error_bounds}

To isolate the spectral mechanism of deflation, we analyze an idealized setting: the deflated components are exact, and both filters are the classical Newton--Schulz family, of which \eqref{eq:muon_ns} is the degree-five member---the polar iteration composes it $L$ times, and the PSD projection uses the same composition as the sign filter $g$.
In practice the deflated components are estimated, but the estimators of \S\ref{sec:deflated_muon} and \S\ref{sec:deflated_psd} are accurate enough to recover the head singular triplets and eigenpairs: Appendix~\ref{app:brsvd} measures their errors against exact decompositions on the matrices of \S\ref{sec:experiments} and finds them small.
The degree-$(2d+1)$ classical Newton--Schulz map~\citep{bjorck1971iterative} is
\begin{equation}
\mathcal{N}_d(X)
:=
a_0X+a_1(XX^{\top})X+a_2(XX^{\top})^2X+\cdots+a_d(XX^{\top})^dX,
\qquad d\geq1,
\label{eq:general_ns_iteration}
\end{equation}
with the Bj\"orck--Bowie coefficients $a_0,\ldots,a_d$ given explicitly in Appendix~\ref{app:proof_deflated_ns}.
The corresponding iteration is $X_{t+1}=\mathcal{N}_d(X_t)$; the degree-five map \eqref{eq:muon_ns} is the case $d=2$.

\myParagraph{Deflated polar-factor analysis}
Let $A$ denote the input, $M\in\R^{n\times m}$ for the polar factor or $Z\in\Sn$ for the projection, with the normalization scale of \eqref{eq:filter}: $\theta(M):=\norm{M}_F$ as in \eqref{eq:muon_ns} and $\theta(Z):=\norm{Z}_2$ as in \eqref{eq:baseline_projection}.
Let $\sigma_1\geq\cdots\geq\sigma_q>0$ be the nonzero singular values of $A$; for $Z$, $\sigma_i=\abs{\lambda_i}$ with the eigenvalues ordered by magnitude.
Let $X^{\mathrm{base}}:=A/\theta(A)$ be the standard initialization and $\ell:=\sigma_q/\theta(A)$ its spectral floor.
For $1\leq k<q$, let $A=H_k+R_k$ be the exact head split of \eqref{eq:identities}, with $H_k=U_k\Sigma_kV_k^{\top}$ for $M$ and $H_k=U_k\Lambda_kU_k^{\top}$ for $Z$, and for $\gamma\geq1$ let
\begin{equation}
X^{\mathrm{def}}_{\gamma}:=\gamma^{-1}\operatorname{polar}(H_k)+R_k/\theta(R_k),
\qquad
\alpha_k:=\frac{\theta(A)}{\theta(R_k)},
\qquad
\varphi_k:=\min\bigl(\alpha_k\ell,\ \gamma^{-1}\bigr),
\label{eq:deflated_init}
\end{equation}
be the padded deflated initialization of Algorithms~\ref{alg:deflated_muon} and~\ref{alg:deflation_admm}, the tail amplification factor of \eqref{eq:tail_amplification}, and the spectral floor of $X^{\mathrm{def}}_{\gamma}$; here $\operatorname{polar}(H_k)$ is $U_kV_k^{\top}$ for $M$ and $U_k\operatorname{sign}(\Lambda_k)U_k^{\top}$ for $Z$, and $\alpha_k\ell=\sigma_q/\theta(R_k)$ is the floor of the normalized residual.
For $M$, $\alpha_k>1$ always; for $Z$, $\alpha_k=\sigma_1/\sigma_{k+1}$.
The following theorem quantifies the effect of deflation.

\begin{theorem}[Exact deflation raises the spectral floor]
\label{thm:deflated_ns}
Let $M\in\R^{n\times m}$ be a nonzero matrix, and let $X^{\mathrm{base}}$, $X^{\mathrm{def}}_{\gamma}$, $\ell$, $\alpha_k$, and $\varphi_k$ be defined as above with $A=M$.
Then $\sigma(X^{\mathrm{base}})\subset[\ell,1]$ and $\sigma(X^{\mathrm{def}}_{\gamma})\subset[\varphi_k,1]$.
Moreover, if $\gamma^{-1}>\ell$, then for every $d\geq1$ and $L\geq0$ the deflated-to-baseline error ratio satisfies
\begin{equation}
\frac{
\bigl\|
\operatorname{polar}(M)-\mathcal{N}_d^{\circ L}\bigl(X^{\mathrm{def}}_{\gamma}\bigr)
\bigr\|_2
}{
\bigl\|
\operatorname{polar}(M)-\mathcal{N}_d^{\circ L}\bigl(X^{\mathrm{base}}\bigr)
\bigr\|_2
}
\leq
\left(
\frac{1-\varphi_k^2}{1-\ell^2}
\right)^{(d+1)^L}
<1.
\label{eq:deflated_ns_ratio}
\end{equation}
\end{theorem}

The proof is deferred to Appendix~\ref{app:proof_deflated_ns}, which also bounds the deflated error itself by $(1-\varphi_k^2)^{(d+1)^L}$, eq.~\eqref{eq:deflated_ns_absolute}.
In words, deflation multiplies the spectral floor by $\alpha_k$ (capped by the padding level $\gamma^{-1}$), and the error ratio decays doubly exponentially in $L$; with $\gamma=1$, $\varphi_k=\alpha_k\ell$ and the unpadded case is recovered.

\myParagraph{Deflated PSD-projection analysis}
We now take $A=Z$, so that $\theta(Z)=\norm{Z}_2=\sigma_1$, $\theta(R_k)=\norm{R_k}_2=\sigma_{k+1}$, and $\alpha_k=\sigma_1/\sigma_{k+1}$; we assume $\sigma_{k+1}<\sigma_1$, so that $\alpha_k>1$.
For a symmetric $X$, the map \eqref{eq:general_ns_iteration} applies the odd polynomial $p_d(x)=\sum_{i=0}^{d}a_ix^{2i+1}$ to every eigenvalue, and $p_d^{\circ L}(x)\to1$ on $(0,1]$, so $g=\mathcal{N}_d^{\circ L}$ is a sign filter for \eqref{eq:baseline_projection} and Algorithm~\ref{alg:deflation_admm}.
Write $\widehat{\Pi}^{\mathrm{base}}(Z):=\half Z\bigl(I+\mathcal{N}_d^{\circ L}(X^{\mathrm{base}})\bigr)$ for \eqref{eq:baseline_projection} with this filter and $\theta=\theta(Z)$, and $\widehat{\Pi}^{\mathrm{def}}_{\gamma}(Z):=U_k(\Lambda_k)_{+}U_k^{\top}+\half R_k\bigl(I+\mathcal{N}_d^{\circ L}(X^{\mathrm{def}}_{\gamma})\bigr)$ for the output of Algorithm~\ref{alg:deflation_admm} with this filter and exact inputs.

\begin{proposition}[Deflated PSD projection]
\label{prop:deflated_psd_ns}
Let $Z\in\Sn$ be a nonzero matrix with $\sigma_{k+1}<\sigma_1$, and let $X^{\mathrm{base}}$, $X^{\mathrm{def}}_{\gamma}$, $\ell$, $\alpha_k$, $\widehat{\Pi}^{\mathrm{base}}$, and $\widehat{\Pi}^{\mathrm{def}}_{\gamma}$ be defined as above with $A=Z$.
Then $\sigma(X^{\mathrm{base}})\subset[\ell,1]$ and $\sigma(X^{\mathrm{def}}_{\gamma})\subset\{\gamma^{-1}\}\cup[\alpha_k\ell,1]$, $\widehat{\Pi}^{\mathrm{def}}_{\gamma}(Z)$ does not depend on $\gamma$, and for every $d\geq1$ and $L\geq0$ the deflated-to-baseline error ratio satisfies
\begin{equation}
\frac{
\bigl\|
\widehat{\Pi}^{\mathrm{def}}_{\gamma}(Z)-\Pi_{\psd{n}}(Z)
\bigr\|_2
}{
\bigl\|
\widehat{\Pi}^{\mathrm{base}}(Z)-\Pi_{\psd{n}}(Z)
\bigr\|_2
}
\leq
\left(
\frac{1-\alpha_k^2\ell^2}{1-\ell^2}
\right)^{(d+1)^L}
<1.
\label{eq:deflated_psd_ratio}
\end{equation}
\end{proposition}

The proof (Appendix~\ref{app:deflated_admm_error}) follows the three steps of the proof of Theorem~\ref{thm:deflated_ns}: the error is diagonal in the eigenbasis of $Z$, with entry $-\frac{\sigma_i}{2}\bigl(1-p_d^{\circ L}(\sigma_i/\theta)\bigr)$ in the direction of every eigenvalue that passes through the filter ($\theta$ the scale in use), and Lemma~\ref{lem:iterated_ns_residual} bounds this residual at the two floors.
In words, the projection obeys the ratio bound of Theorem~\ref{thm:deflated_ns} at $\gamma=1$: deflation replaces the condition number of $Z$ by that of the residual.
Two differences from Theorem~\ref{thm:deflated_ns} are worth noting.
First, the padded head is annihilated by $R_k$ and never reaches the output, so the bounds involve the residual floor $\alpha_k\ell$ rather than $\varphi_k$, and no side condition is needed beyond $\alpha_k>1$, which the gate guarantees under exact spectral estimates.
Second, the error in each eigendirection is weighted by $\sigma_i$: small eigenvalues that the filter cannot resolve are cheap for the projection, whereas they dominate the error of the polar factor.
Appendix~\ref{app:deflated_admm_error} also bounds the deflated error itself, eq.~\eqref{eq:deflated_psd_absolute}, and gives complementary bounds for the implemented fixed filter (Proposition~\ref{prop:fixed_filter}) and for an idealized minimax filter regenerated on the deflated interval (Theorem~\ref{thm:deflated_psd_error}).

%!TEX root = ../main.tex

\section{Numerical Experiments}
\label{sec:experiments}

\myParagraph{General setup}
All language-model experiments pretrain GPT-2 Large~\citep{radford2019language} ($n_{\mathrm{embd}}=1280$, $n_{\mathrm{layer}}=36$, $n_{\mathrm{head}}=20$, vocabulary size $50{,}257$, context length $1024$) on subsets of FineWeb~\citep{penedo2024fineweb}, in \texttt{bfloat16} mixed precision on four NVIDIA H200 GPUs; the SDP experiments run on a single H200 with CUDA 12.9.
Muon is applied to all eligible parameters, excluding token embeddings, positional embeddings, and the output embedding layer; the remaining parameters are optimized with AdamW~\citep{loshchilov2019decoupled}.
The learning rate is held constant for the first $40\%$ of training and then decays linearly to zero, following~\citet{jordan2024muon}; the ten-billion-token runs shorten the constant phase to the first $5\%$ for stability.
\S\ref{sec:deflation} analyzes the classical Newton--Schulz coefficients; the experiments use the two practical mappings: Muon's retuned degree-five coefficients $a=3.4445$, $b=-4.7750$, $c=2.0315$ with $L=5$~\citep{jordan2024muon,cesista2025coeffs}, whose larger slope at the origin amplifies small singular values faster, and Polar Express, which selects iteration-dependent degree-five polynomials by a minimax criterion under the same multiplication budget~\citep{amsel2025polar}.
We refer to them as Newton--Schulz and Polar Express below; deflation changes neither mapping's coefficients.
Unless stated otherwise, deflation uses the window ratio $w=0.025$, threshold $\tau=0.1$, and padding $\gamma=1.01$ ($\gamma=1.1$ for the Polar Express mapping).

\subsection{Matrix-Function Accuracy}
\label{subsec:convergence_experiments}

We train GPT-2 Large with Muon on 1B tokens from FineWeb and record the momentum matrices of the four attention projections $\mathrm{Q}$, $\mathrm{K}$, $\mathrm{V}$, $\mathrm{O}$ and the two MLP projections in the middle transformer block (layer $\lfloor N/2\rfloor$ of $N$ layers), at training steps $1$, $500$, $1000$, and $1500$.
For each momentum matrix, we compare four methods with the entries the optimizer uses: Muon's Newton--Schulz iteration, Polar Express, and their deflated counterparts, whose head is estimated by the batched randomized SVD of Algorithm~\ref{alg:brsvd} with the gate and constants of the training runs, not read off an exact SVD.
Using $UV^{\top}$ from a double-precision SVD as the reference polar factor, we measure the relative Frobenius error $\norm{X_L-UV^{\top}}_F/\norm{UV^{\top}}_F$ after every iteration; the polynomial iterations are evaluated in single precision, so that the effects of estimation and of the mapping are not confounded with the \texttt{bfloat16} rounding that all four arms share (Appendix~\ref{app:additional} quantifies the latter).
Figure~\ref{fig:convergence_attn_v} shows that when the momentum matrix has a few dominant singular values, the deflated variants attain the lower error at every iteration. 
Appendix~\ref{app:additional} reports all six recorded matrices (Figures~\ref{fig:convergence_mlp_all} and~\ref{fig:convergence_attn_all}).

\begin{figure}[h]
\begin{center}
\includegraphics[width=0.94\linewidth]{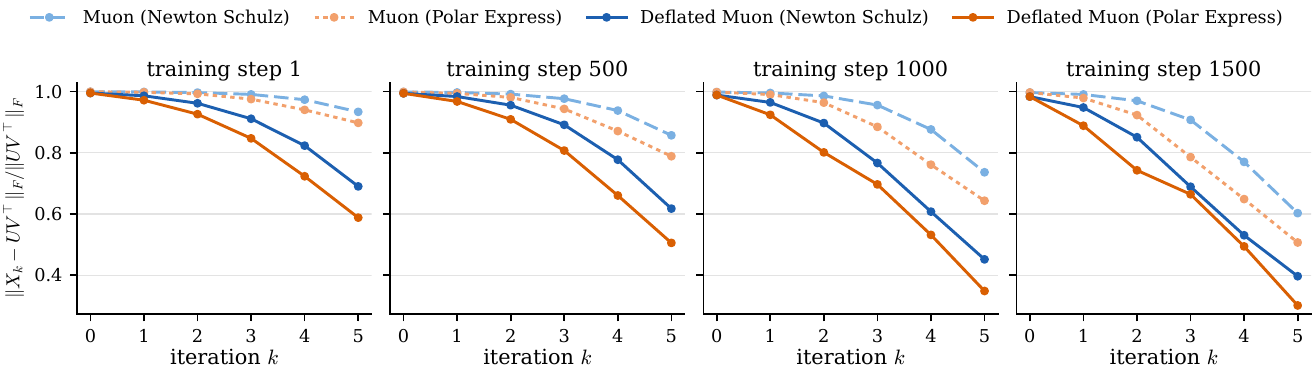}
\end{center}
\caption{Matrix-level accuracy on the attention $\mathrm{V}$ momentum matrix of
GPT-2 Large, recorded at training steps $1$, $500$, $1000$, and $1500$, as a
function of the number of iterations. Each panel reports the relative
error $\norm{X_L-UV^{\top}}_F/\norm{UV^{\top}}_F$.}
\label{fig:convergence_attn_v}
\end{figure}

\subsection{Language-Model Pretraining}
\label{sec:main_experiments}

We integrate \textsc{Deflated Muon} into full pretraining: one epoch over a one-billion-token subset of FineWeb.
Figures~\ref{fig:intro_loss} and~\ref{fig:lr_sweep} report the results. 
In both cases, \textsc{Deflated Muon} achieves better validation loss than Muon, and the improvement is sustained over the whole run. Although the batched randomized SVD and deflation add per-step computation, the advantage persists against wall-clock time.
\textsc{Deflated Muon} attains lower final validation loss than Muon at most learning rates for both mappings, and at every learning rate for Newton--Schulz.

\begin{figure}[h]
\centering
% Widths chosen so that both panels come out the same height: the learning-rate
% panel is 469 x 224 pt (aspect 2.10), the loss panel 441 x 171 pt (aspect 2.58),
% so 0.443/2.10 = 0.547/2.58; the boxes sum to 0.99\linewidth to keep the gap small.
\begin{subfigure}[b]{0.443\linewidth}
\centering
\includegraphics[width=\linewidth]{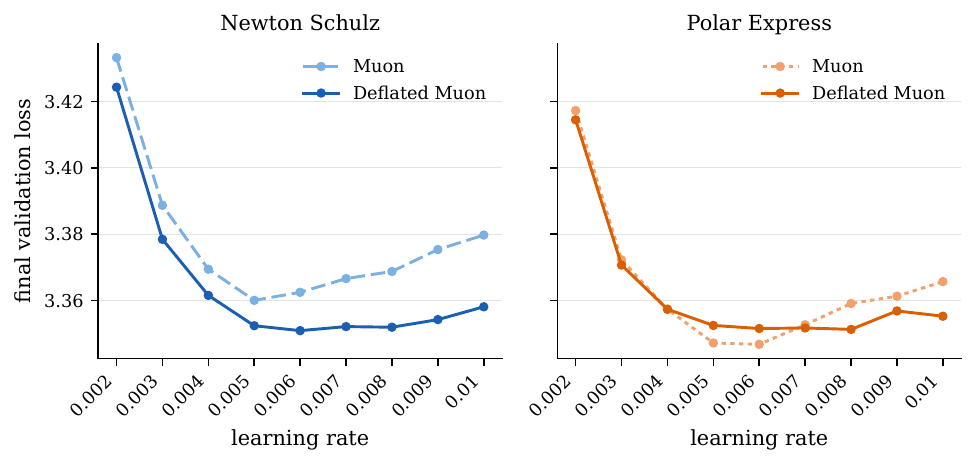}
\caption{Final validation loss across learning rates}
\label{fig:lr_sweep_panel}
\end{subfigure}\hfill
\begin{subfigure}[b]{0.547\linewidth}
\centering
\includegraphics[width=\linewidth]{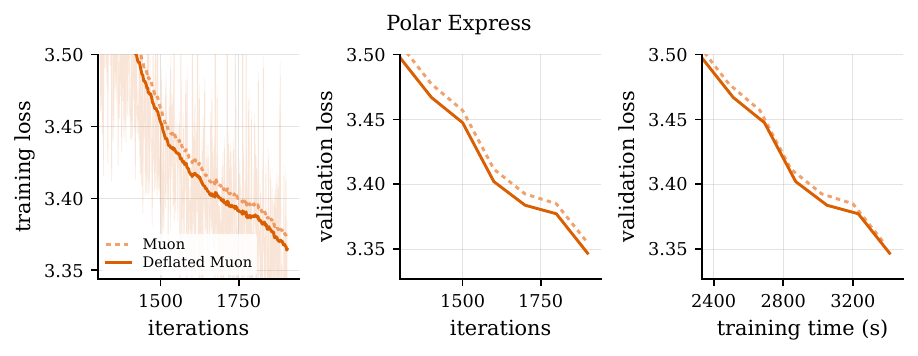}
\caption{Polar Express mapping, learning rate $0.01$}
\label{fig:pe_loss}
\end{subfigure}
\caption{GPT-2 Large trained on one billion tokens from FineWeb.
(a) Final validation loss across learning rates, for both mappings.
(b) Polar Express loss curves at learning rate $0.01$; layout as in Figure~\ref{fig:intro_muon}.}
\label{fig:lr_sweep}
\end{figure}

\myParagraph{Longer training horizon}
To test whether the benefit persists over a longer horizon, we train GPT-2 Large on ten billion tokens, roughly the Chinchilla-optimal budget for this model size~\citep{hoffmann2022training}.
As shown in Figure~\ref{fig:10b_loss}, \textsc{Deflated Muon} continues to outperform the original Muon, reaching a lower final validation loss.

\begin{figure}[H]
\centering
\begin{subfigure}[t]{0.49\linewidth}
\centering
\includegraphics[width=\linewidth]{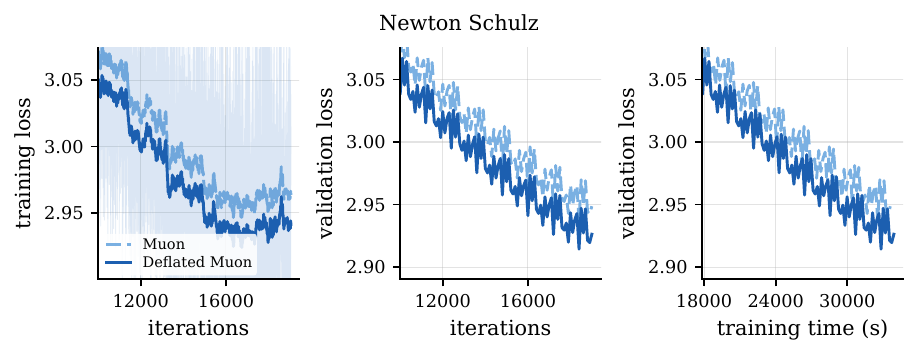}
\caption{Newton--Schulz mapping.}
\label{fig:10b_loss_ns}
\end{subfigure}\hfill
\begin{subfigure}[t]{0.49\linewidth}
\centering
\includegraphics[width=\linewidth]{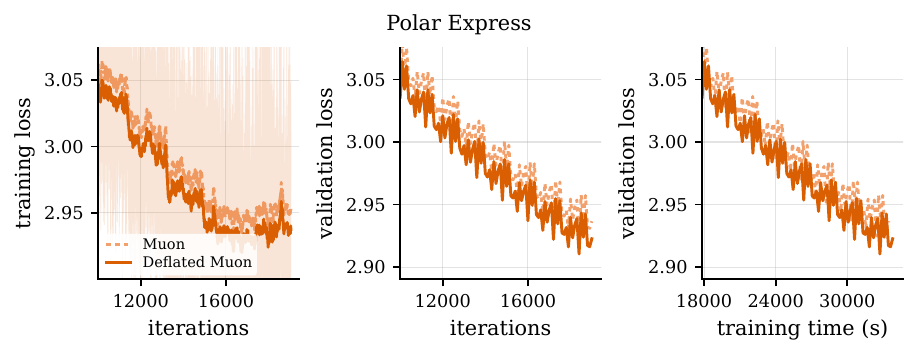}
\caption{Polar Express mapping.}
\label{fig:10b_loss_pe}
\end{subfigure}
\caption{GPT-2 Large trained on ten billion tokens from FineWeb, for the Newton--Schulz (a) and Polar Express (b) mappings; layout as in Figure~\ref{fig:intro_muon}.}
\label{fig:10b_loss}
\end{figure}

\subsection{Semidefinite Programming}
\label{sec:sdp_experiments}

We next evaluate the deflated PSD projection inside ADMM on four large instances from Mittelmann's sparse SDP collection~\citep{mittelmann2006sparse}: the max-cut relaxations~\citep{goemans1995maxcut} \texttt{G55mc} and \texttt{G59mc} ($n=5{,}000$) and \texttt{G60mc} ($n=7{,}000$), and the min-bisection relaxation \texttt{G60\_mb} ($n=7{,}000$).
We run the ADMM iteration \eqref{eq:admm_iter} with $\sigma=1$ for a fixed budget of $10^4$ iterations, with every PSD projection performed in FP16, and compare the factorization-free projection \eqref{eq:baseline_projection} against its deflated version (Algorithm~\ref{alg:deflation_admm}), everything else being identical.
The deflated projection uses $\tau=0.1$, $w=0.025$, $w'=0.05$, and $\gamma=1.1$, with the gate off for the first $100$ iterations while the tracked basis stabilizes.
We monitor the eigenvalue-free part of the maximum KKT residual $\eta$ (definition and end-of-run PSD-feasibility checks in Appendix~\ref{app:admm_setup}).

Figure~\ref{fig:admm_convergence} reports $\eta$ against cumulative projection time: across the four instances, ADMM with the deflated projection converges to a lower KKT residual within a similar projection-time budget.
\begin{figure}[H]
\begin{center}
\includegraphics[width=0.96\linewidth]{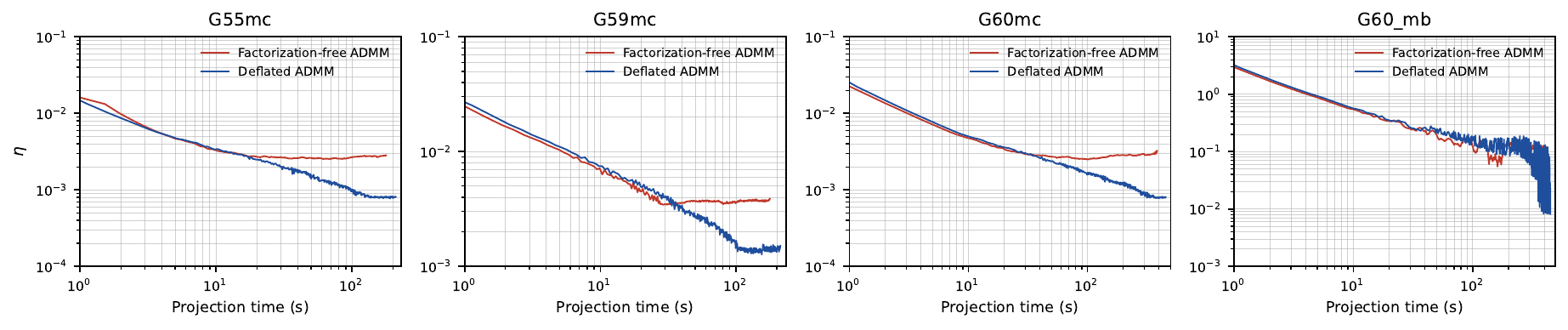}
\end{center}
\caption{KKT residual $\eta$ (the three eigenvalue-free terms of
\eqref{eq:admm_eta}) against cumulative projection time on four instances from
Mittelmann's collection, with the setup of \S\ref{sec:sdp_experiments}. We
compare the factorization-free ADMM (red) with ADMM using the deflated
projection (blue).}
\label{fig:admm_convergence}
\end{figure}

\subsection{Ablations}
\label{sec:ablations}

\myParagraph{Number of polynomial iterations}
We vary the number of iterations over $L\in\{0,1,2,3,4,5\}$ (Table~\ref{tab:ns_iterations}): \textsc{Deflated Muon} is less sensitive to $L$ and maintains low validation loss even with fewer iterations; notably, with Polar Express, $L=3$ deflated iterations (3.3493) already outperform the $L=5$ baseline (3.3671).

\begin{table}[H]
\centering
\footnotesize
\setlength{\tabcolsep}{4pt}
\begin{tabular}{llcccccc}
\toprule
 & & \multicolumn{6}{c}{Number of iterations $L$} \\
\cmidrule(l){3-8}
Mapping & Method & $0$ & $1$ & $2$ & $3$ & $4$ & $5$ \\
\midrule
Newton--Schulz & Muon                   & $4.1879$ & $3.8746$ & $3.6646$ & $3.5201$ & $3.4084$ & $3.3794$ \\
               & \textsc{Deflated Muon} & $\mathbf{4.0898}$ & $\mathbf{3.7895}$ & $\mathbf{3.5870}$ & $\mathbf{3.4400}$ & $\mathbf{3.3681}$ & $\mathbf{3.3568}$ \\
\midrule
Polar Express  & Muon                   & $4.1916$ & $3.7698$ & $3.5122$ & $3.4244$ & $3.3770$ & $3.3671$ \\
               & \textsc{Deflated Muon} & $\mathbf{4.0805}$ & $\mathbf{3.6717}$ & $\mathbf{3.4814}$ & $\mathbf{3.3493}$ & $\mathbf{3.3545}$ & $\mathbf{3.3604}$ \\
\bottomrule
\end{tabular}
\caption{Final validation loss of GPT-2 Large as a function of the number of polynomial iterations $L$.
For Polar Express, with deflation, $L=3$ is already sufficient to reach the best validation loss.}
\label{tab:ns_iterations}
\end{table}
\myParagraph{Momentum coefficient}
We evaluate the momentum coefficient $\beta\in\{0.90,0.95,0.98\}$ on the one-billion-token setup of \S\ref{sec:main_experiments}.
As shown in Table~\ref{tab:momentum}, the improvement from deflation is robust across all three values, and Figure~\ref{fig:mom090_loss} shows the results for $\beta=0.90$ with the Newton--Schulz mapping.
% Table and figure side by side: the table in the left minipage (its caption
% via \captionof, from the caption package that subcaption loads), the loss
% curves in the right one; [c] aligns the two vertically centred.
% The figure is the half-width rendering of figures/momentum_ablation/
% gpt-large_mom0.90_ns_val_loss.pdf (same data, larger fonts relative to
% the panels), produced with
%   python plotting/plot_momentum_ablation.py gpt-large --figwidth 8 --aspect 0.7 --out <dir>
\begin{figure}[H]
\vspace{-4pt}
\setlength{\abovecaptionskip}{3pt}
\setlength{\belowcaptionskip}{0pt}
\begin{minipage}[c]{0.45\linewidth}
\centering
\scriptsize
\setlength{\tabcolsep}{4pt}
% the mapping is a group header row rather than a column, to fit the half width
\begin{tabular}{lccc}
\toprule
 & \multicolumn{3}{c}{Momentum coefficient $\beta$} \\
\cmidrule(l){2-4}
Method & $0.90$ & $0.95$ & $0.98$ \\
\midrule
\multicolumn{4}{l}{\emph{Newton--Schulz}} \\
Muon                   & $3.4205$ & $3.3803$ & $3.3875$ \\
\textsc{Deflated Muon} & $\mathbf{3.3845}$ & $\mathbf{3.3599}$ & $\mathbf{3.3675}$ \\
\midrule
\multicolumn{4}{l}{\emph{Polar Express}} \\
Muon                   & $3.3967$ & $3.3672$ & $3.3786$ \\
\textsc{Deflated Muon} & $\mathbf{3.3776}$ & $\mathbf{3.3576}$ & $\mathbf{3.3785}$ \\
\bottomrule
\end{tabular}
\captionof{table}{Final validation loss of \mbox{GPT-2} Large for momentum coefficients $\beta$; \textsc{Deflated Muon} outperforms Muon at every $\beta$.}
\label{tab:momentum}
\end{minipage}\hfill
\begin{minipage}[c]{0.53\linewidth}
\centering
\includegraphics[width=\linewidth]{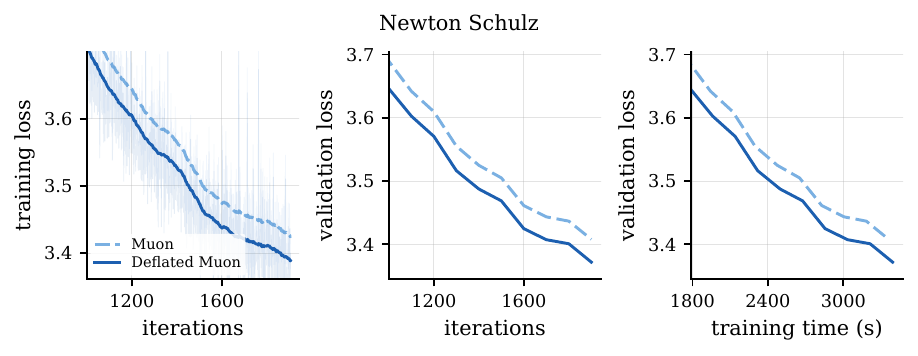}
\caption{GPT-2 Large with momentum coefficient $\beta=0.90$, Newton--Schulz mapping; layout as in Figure~\ref{fig:intro_muon}.}
\label{fig:mom090_loss}
\end{minipage}
\vspace{-8pt}
\end{figure}

\myParagraph{Deflation ratio}
\textsc{Deflated Muon} has two constants of its own: the window ratio $w$, which determines how deep into the spectrum the gate looks, and the threshold $\tau$, which is shared by the gate and the rank cutoff (\S\ref{sec:deflated_muon}).
Starting from the operating point $(w,\tau)=(0.025,0.1)$ used in the main experiments, we test $w\in\{0.0025,0.005,0.01,0.025\}$ with $\tau=0.1$, and $\tau\in\{0.05,0.1,0.2,0.3\}$ with $w=0.025$.
Tables~\ref{tab:gate_w} and~\ref{tab:gate_tau} (Appendix~\ref{app:additional}) show that these settings affect the size of the deflated head: the number of deflated singular pairs decreases as we narrow the window or raise the threshold, while the improvement in validation loss remains robust across different settings.

\myParagraph{Deflation cutoff}
Momentum spectra generally become less head-dominated as training progresses~\citep{magakyan2026spectral}, so deflation may not be needed throughout training; we therefore vary the cutoff over $\{10\%,30\%,50\%,100\%\}$ of the optimizer steps (the main experiments use $100\%$).
Table~\ref{tab:cutoff} (Appendix~\ref{app:additional}) shows that \textsc{Deflated Muon} improves on the baseline at every tested cutoff; a cutoff of $50\%$ or even $30\%$ already substantially improves performance.

%!TEX root = ../main.tex

\section{Conclusion}
\label{sec:conclusion}

We introduced spectral deflation, a general framework that preserves the target exactly while raising the spectral floor of the input and therefore improves the accuracy of polynomial approximations.
Integrated into Muon, \textsc{Deflated Muon} consistently improves GPT-2 pretraining over the corresponding Muon baselines; integrated into factorization-free ADMM for SDP, the deflated PSD projection substantially reduces the KKT residual the solver converges to within a similar projection-time budget.
The framework is orthogonal to the design of the polynomial, and natural next instances are inverse matrix roots in preconditioned optimizers~\citep{gupta2018shampoo} and other spectral functions computed by fixed-depth filters.

%!TEX root = ../main.tex

\subsection*{AI use statement}

In this work, we used generative AI tools for language polishing and code assistance.
We did not use generative AI for research ideation, experimental design, or experimental data generation.
All LLM-generated code was verified and tested for correctness by the authors.
We take full responsibility for the final content of this work, including all text, claims, code, and other artifacts produced with the aid of generative AI.

\subsection*{Reproducibility statement}

The two primitives are fully specified in \S\ref{sec:deflation} (Algorithms~\ref{alg:deflated_muon}--\ref{alg:deflation_admm}), with implementation details in Appendices~\ref{app:brsvd} and~\ref{app:deflated_admm}.
Details of the experiments, including the model architecture, training data, hyperparameters, and evaluation metrics, are provided in \S\ref{sec:experiments} and Appendix~\ref{app:admm_setup}.
Complete proofs of the theoretical results are provided in Appendices~\ref{app:proof_deflated_ns} and~\ref{app:deflated_admm_error}.
Source code for all experiments is available at \url{https://github.com/ComputationalRobotics/Spectral-Deflation}.

\bibliographystyle{plainnat}
\bibliography{references}

% ---------------------------------------------------------------------------
% Appendix.  Each appendix section lives in its own file under sections/.
% ---------------------------------------------------------------------------
\appendix
%!TEX root = ../main.tex

\section{Related Work}
\label{sec:related_work}

\paragraph{Muon and orthogonalized-update optimizers.}
Muon constructs updates for matrix-valued parameters by approximating the polar factors of their momentum matrices with a small, fixed number of Newton--Schulz iterations~\citep{jordan2024muon,liu2025muonscalable}.
\citet{bernstein2024old} interpret the update as steepest descent under spectral-norm geometry, and \citet{chen2025muon} analyze Muon through implicit spectral-norm constraints.
Subsequent work modifies either the output of orthogonalization or the coordinates in which it is performed.
Among post-orthogonalization methods, Muon+ adds row- or column-wise normalization to correct norm imbalance in the polar update~\citep{zhang2026muonplus}, and NorMuon combines orthogonalization with neuron-wise second-moment normalization~\citep{li2025normuon}; AdaMuon is a hybrid, applying a sign-stabilized momentum transformation before orthogonalization and element-wise adaptive scaling after it~\citep{si2025adamuon}.
More broadly, geometry-aware preconditioners motivate working in transformed coordinates: SOAP performs Adam-style adaptation in an eigenbasis obtained from Shampoo preconditioners~\citep{gupta2018shampoo,vyas2024soap}.
Pre-orthogonalization corrections modify the matrix seen by the orthogonalization polynomial: FISMO and Mousse compute spectral updates in curvature-whitened coordinates~\citep{xu2026fismo,zhang2026mousse}, and MuonEq applies lightweight row and column equilibration before Newton--Schulz~\citep{chang2026muoneq}.
These methods deliberately change the update's scaling or geometry, and hence its exact target.
In contrast, exact spectral deflation preserves the original polar target and aims specifically at the finite-step approximation error.

\paragraph{Polar-factor computation.}
The polar factor can be computed accurately through an SVD, but a full factorization for every matrix at every optimizer step is expensive~\citep{higham1986polar}.
Classical iterative alternatives include Newton, inverse Newton, Newton--Schulz, and Halley-type methods for the polar decomposition or the matrix sign function~\citep{higham2008functions,nakatsukasa2010halley,bjorck1971iterative}.
Prioritizing speed over asymptotic accuracy, Muon uses a retuned degree-five Newton--Schulz map for only a few iterations; the retuned coefficients accelerate tail singular values although one is no longer a fixed point of the scalar map~\citep{jordan2024muon,cesista2025coeffs}.
Recent work improves this trade-off by redesigning the polynomials: CANS derives Chebyshev-type coefficients for controlled approximate orthogonalization~\citep{grishina2025cans}, and Polar Express selects iteration-dependent minimax polynomials under a fixed multiplication budget~\citep{amsel2025polar}.
A complementary line reduces the cost of the polar computation itself: \citet{choudhury2026muon} analyze Muon with a randomized low-rank polar approximation, and CacheMuon reuses temporal information across optimizer steps~\citep{dev2026cachemuon}.
Deflation is orthogonal to both lines: it changes the normalized spectrum presented to an existing filter, keeps the full-rank target through the exact identity \eqref{eq:identities}, and we evaluate it with both Newton--Schulz and Polar Express; the temporal reuse of CacheMuon parallels our warm-started tracking on the ADMM side.

\paragraph{ADMM and semidefinite programming.}
ADMM is a widely used operator-splitting framework whose low per-iteration cost and decomposable structure make it suitable for large-scale constrained optimization~\citep{boyd2011admm}.
For semidefinite programming, \citet{wen2010admm} developed an alternating-direction augmented-Lagrangian method in which each iteration consists of linear-system operations and a projection onto the PSD cone; related first-order solvers include the homogeneous self-dual splitting of SCS~\citep{odonoghue2016scs} and ADMM for sparse SDPs~\citep{zheng2017fast}, with recent GPU implementations for large multi-block problems~\citep{groudiev2025cuadmm}.
Because the PSD projection~\citep{higham1988nearest} is conventionally computed by a full eigendecomposition, it often dominates the cost for large dense instances.
Approximate ADMM reduces this cost with partial eigendecompositions~\citep{rontsis2022efficient}, low-rank operator-splitting methods exploit low-rank solution structure~\citep{souto2022lowrank}, and fixed-point methods approximate the projection through the matrix sign function~\citep{francisco2017fixed}.
Most recently, \citet{kang2025psdprojection} designed a factorization-free composite polynomial filter that approximates the matrix ReLU using dense GEMMs and low-precision arithmetic; we retain that filter and improve the normalized spectrum it receives.

\paragraph{Spectral deflation and spectrum-aware filtering.}
Deflation is a classical acceleration technique for linear systems, eigenvalue problems, and matrix-function approximation: selected invariant components are handled explicitly, and an iterative approximation is applied only to their complement.
Deflated conjugate gradient removes the outlying eigenvalues that slow convergence~\citep{saad2000deflated}; deflated restarting accelerates Krylov approximations of matrix functions~\citep{eiermann2011deflated}; and deflated rational Krylov methods for the matrix sign function explicitly treat eigenvalues near the function's discontinuity~\citep{bloch2009krylov}.
Other approaches adapt the approximation to the spectrum instead: spectrum-adapted polynomial methods estimate the spectral density and allocate accuracy to densely populated regions~\citep{fan2018spectrum}.
We adapt the deflation principle to short, dense-GEMM polynomial filters, where the dominant components are removed not because they impede a Krylov expansion but because they inflate the normalization scale of the fixed-depth filter.

%!TEX root = ../main.tex

\section{Proof of Theorem~\ref{thm:deflated_ns}}
\label{app:proof_deflated_ns}

We restate the theorem for self-containment, adding the bound on the deflated error itself that the proof yields alongside the ratio bound.

\begin{theoremrestated}
Let $M$, $X^{\mathrm{base}}$, $X^{\mathrm{def}}_{\gamma}$, $\ell$, $\alpha_k$, and $\varphi_k=\min(\alpha_k\ell,\gamma^{-1})$ be as defined in \S\ref{subsec:error_bounds}, with $1\leq k<q$ and $\gamma\geq1$.
Then $\sigma(X^{\mathrm{base}})\subset[\ell,1]$ and $\sigma(X^{\mathrm{def}}_{\gamma})\subset[\varphi_k,1]$.
Moreover, for every $d\geq1$ and $L\geq0$,
\begin{equation}
\bigl\|
\operatorname{polar}(M)
-\mathcal{N}_d^{\circ L}(X^{\mathrm{def}}_{\gamma})
\bigr\|_2
\leq
\bigl(1-\varphi_k^2\bigr)^{(d+1)^L}.
\label{eq:deflated_ns_absolute}
\end{equation}
If in addition $\gamma^{-1}>\ell$, the deflated-to-baseline error ratio satisfies
\[
\frac{
\bigl\|
\operatorname{polar}(M)
-\mathcal{N}_d^{\circ L}(X^{\mathrm{def}}_{\gamma})
\bigr\|_2
}{
\bigl\|
\operatorname{polar}(M)
-\mathcal{N}_d^{\circ L}(X^{\mathrm{base}})
\bigr\|_2
}
\leq
\left(
\frac{1-\varphi_k^2}{1-\ell^2}
\right)^{(d+1)^L}
<1.
\]
\end{theoremrestated}

We prove the theorem in three steps.
First, two scalar lemmas describe how one Newton--Schulz step, and then $L$ steps, act on a squared singular value.
Second, we reduce the matrix error to a scalar quantity at the spectral floor.
Third, we apply the lemmas at the floors $\ell$ and $\varphi_k$.
The first lemma gives an exact formula for the one-step squared-singular-value residual, and the second iterates it.
Throughout, the Bj\"orck--Bowie coefficients of \eqref{eq:general_ns_iteration} are $a_i:=(-1)^i\sum_{j=i}^{d}\frac{1}{4^j}\binom{2j}{j}\binom{j}{i}$ for $i=0,\ldots,d$.

\begin{lemma}[One-step residual identity]
\label{lem:one_step_ns_residual}
For $d\geq1$, let
$c_j:=4^{-j}\binom{2j}{j}$ and
$h_d(z):=\sum_{j=0}^d c_jz^j$, and let
$p_d(x):=\sum_{i=0}^d a_i x^{2i+1}=x\,h_d(1-x^2)$
be the scalar polynomial induced by $\mathcal{N}_d$. Then $p_d$ is
nondecreasing and maps $[0,1]$ into itself. Moreover,
\begin{equation}
1-p_d(x)^2
=
(1-x^2)^{d+1}\omega_d(1-x^2),
\label{eq:one_step_ns_residual}
\end{equation}
where
\begin{equation}
\omega_d(z)
:=
(2d+1)c_d
\sum_{j=0}^d\frac{c_j}{d+j+1}z^j.
\label{eq:omega_d_definition}
\end{equation}
The function $\omega_d$ is positive and nondecreasing on $[0,1]$ and
satisfies $0<\omega_d(z)\leq1$ for every $z\in[0,1]$. Consequently,
\begin{equation}
0\leq 1-p_d(x)^2\leq(1-x^2)^{d+1}
\qquad\text{for every }x\in[0,1].
\label{eq:one_step_ns_residual_bound}
\end{equation}
\end{lemma}

\begin{proof}
We first differentiate $p_d$. The coefficients $c_j$ satisfy
$c_{j+1}/c_j=\tfrac14\binom{2j+2}{j+1}/\binom{2j}{j}=\tfrac{2j+1}{2j+2}$,
and hence
\begin{equation}
2(j+1)c_{j+1}=(2j+1)c_j.
\label{eq:central_binomial_recurrence}
\end{equation}
Set $z=1-x^2$. By the chain rule,
$p_d'(x)=h_d(z)-2(1-z)h_d'(z)$. Expanding the right-hand side gives
\begin{align*}
p_d'(x)
&=
\sum_{j=0}^d c_jz^j
-2(1-z)\sum_{j=1}^d jc_jz^{j-1} \\
&=
(c_0-2c_1)
+\sum_{j=1}^{d-1}
\bigl((2j+1)c_j-2(j+1)c_{j+1}\bigr)z^j
+(2d+1)c_dz^d.
\end{align*}
Because $c_0=1$, $c_1=1/2$, and
\eqref{eq:central_binomial_recurrence} holds, every term except the
last one vanishes. Thus $p_d'(x)=(2d+1)c_d(1-x^2)^d$. In particular, $p_d'(x)\geq0$ on $[0,1]$. Since $p_d(0)=0$ and
$p_d(1)=1$, it follows that $p_d([0,1])\subset[0,1]$.

It remains to prove the residual identity. For $z\in[0,1]$, direct
differentiation and the formula for $p_d'$ give
\begin{align*}
\frac{\mathrm{d}}{\mathrm{d}z}
\bigl[1-(1-z)h_d(z)^2\bigr]
&=
h_d(z)^2-2(1-z)h_d(z)h_d'(z) \\
&=
h_d(z)\bigl[h_d(z)-2(1-z)h_d'(z)\bigr] \\
&=
(2d+1)c_dz^dh_d(z).
\end{align*}
Because $h_d(0)=c_0=1$, the expression inside the derivative is zero
at $z=0$. Integrating from $0$ to $z$ therefore yields
\begin{align*}
1-(1-z)h_d(z)^2
&=
(2d+1)c_d\int_0^z u^dh_d(u)\,\mathrm{d}u \\
&=
(2d+1)c_d
\sum_{j=0}^d c_j\int_0^z u^{d+j}\,\mathrm{d}u \\
&=
(2d+1)c_d
\sum_{j=0}^d\frac{c_j}{d+j+1}z^{d+j+1} \\
&=
z^{d+1}\omega_d(z).
\end{align*}
Substituting $z=1-x^2$ and using
$p_d(x)^2=x^2h_d(1-x^2)^2$ proves
\eqref{eq:one_step_ns_residual}.

Every coefficient of $\omega_d$ is positive. Hence $\omega_d(z)>0$
and $\omega_d$ is nondecreasing on $[0,1]$. Taking $x=0$ in
\eqref{eq:one_step_ns_residual} gives $\omega_d(1)=1$. It follows that
$0<\omega_d(z)\leq1$ for every $z\in[0,1]$.
Combining this inequality with \eqref{eq:one_step_ns_residual} proves
\eqref{eq:one_step_ns_residual_bound}.
\end{proof}

\begin{lemma}[Iterated squared-residual bounds]
\label{lem:iterated_ns_residual}
For an integer $t\geq0$, let $p_d^{\circ t}$ denote the $t$-fold
composition of $p_d$, with $p_d^{\circ0}$ the identity. Then, for every
$x\in[0,1]$ and every $L\geq0$,
\begin{equation}
1-p_d^{\circ L}(x)
\leq
1-\bigl(p_d^{\circ L}(x)\bigr)^2
\leq
(1-x^2)^{(d+1)^L}.
\label{eq:iterated_absolute_bound}
\end{equation}

Moreover, if $0<x\leq y\leq1$ and $x<1$, then
\begin{equation}
\frac{1-p_d^{\circ L}(y)}{1-p_d^{\circ L}(x)}
\leq
\left(
\frac{1-y^2}{1-x^2}
\right)^{(d+1)^L}.
\label{eq:iterated_ratio_bound}
\end{equation}
\end{lemma}

\begin{proof}
Lemma~\ref{lem:one_step_ns_residual} shows that $p_d$ maps $[0,1]$
into itself, so $p_d^{\circ t}(x)\in[0,1]$ for every $t\geq0$.
Applying \eqref{eq:one_step_ns_residual} to $p_d^{\circ t}(x)$ and
writing $p_d^{\circ(t+1)}(x)=p_d\bigl(p_d^{\circ t}(x)\bigr)$ gives
\begin{equation}
1-\bigl(p_d^{\circ(t+1)}(x)\bigr)^2
=
\Bigl(1-\bigl(p_d^{\circ t}(x)\bigr)^2\Bigr)^{d+1}
\omega_d\Bigl(1-\bigl(p_d^{\circ t}(x)\bigr)^2\Bigr).
\label{eq:iterated_residual_recursion}
\end{equation}
Because $0<\omega_d(z)\leq1$ on $[0,1]$,
\begin{equation}
1-\bigl(p_d^{\circ(t+1)}(x)\bigr)^2
\leq
\Bigl(1-\bigl(p_d^{\circ t}(x)\bigr)^2\Bigr)^{d+1}.
\label{eq:residual_power_recursion}
\end{equation}
At $t=0$, we have
$1-\bigl(p_d^{\circ0}(x)\bigr)^2=1-x^2$. Induction using
\eqref{eq:residual_power_recursion} therefore yields
\begin{equation}
1-\bigl(p_d^{\circ L}(x)\bigr)^2
\leq
(1-x^2)^{(d+1)^L}
\qquad\text{for every }L\geq0.
\label{eq:rL_absolute_bound}
\end{equation}
Furthermore, since $p_d^{\circ L}(x)\in[0,1]$,
$1-p_d^{\circ L}(x)=\bigl(1-(p_d^{\circ L}(x))^2\bigr)/\bigl(1+p_d^{\circ L}(x)\bigr)\leq1-\bigl(p_d^{\circ L}(x)\bigr)^2$.
Together with \eqref{eq:rL_absolute_bound}, this proves
\eqref{eq:iterated_absolute_bound}.

We next prove the ratio estimate. By
Lemma~\ref{lem:one_step_ns_residual}, $p_d$ is nondecreasing on
$[0,1]$, and hence so is $p_d^{\circ t}$. Since $x\leq y$ and both
iterates lie in $[0,1]$,
\begin{equation}
p_d^{\circ t}(x)\leq p_d^{\circ t}(y)
\qquad\text{and}\qquad
0\leq 1-\bigl(p_d^{\circ t}(y)\bigr)^2
\leq 1-\bigl(p_d^{\circ t}(x)\bigr)^2
\qquad\text{for every }t\geq0.
\label{eq:residual_ordering}
\end{equation}
Because $x<1$, we have $1-x^2>0$, and since $\omega_d$ is positive,
\eqref{eq:iterated_residual_recursion} shows inductively that
$1-\bigl(p_d^{\circ t}(x)\bigr)^2>0$ for every finite $t$. Hence the
following ratios are well-defined. Applying
\eqref{eq:iterated_residual_recursion} to $x$ and to $y$, and using the
monotonicity of $\omega_d$ together with \eqref{eq:residual_ordering},
we obtain
\begin{align*}
\frac{1-\bigl(p_d^{\circ(t+1)}(y)\bigr)^2}
     {1-\bigl(p_d^{\circ(t+1)}(x)\bigr)^2}
&=
\left(
\frac{1-\bigl(p_d^{\circ t}(y)\bigr)^2}
     {1-\bigl(p_d^{\circ t}(x)\bigr)^2}
\right)^{d+1}
\frac{\omega_d\Bigl(1-\bigl(p_d^{\circ t}(y)\bigr)^2\Bigr)}
     {\omega_d\Bigl(1-\bigl(p_d^{\circ t}(x)\bigr)^2\Bigr)}
\\
&\leq
\left(
\frac{1-\bigl(p_d^{\circ t}(y)\bigr)^2}
     {1-\bigl(p_d^{\circ t}(x)\bigr)^2}
\right)^{d+1}.
\end{align*}
Induction on $t$ now gives
\begin{equation}
\frac{1-\bigl(p_d^{\circ L}(y)\bigr)^2}
     {1-\bigl(p_d^{\circ L}(x)\bigr)^2}
\leq
\left(
\frac{1-y^2}{1-x^2}
\right)^{(d+1)^L}.
\label{eq:rL_ratio_bound}
\end{equation}

Finally, \eqref{eq:residual_ordering} gives
$p_d^{\circ L}(x)\leq p_d^{\circ L}(y)$. Rationalizing the two
differences therefore gives
\[
\frac{1-p_d^{\circ L}(y)}{1-p_d^{\circ L}(x)}
=
\frac{1-\bigl(p_d^{\circ L}(y)\bigr)^2}
     {1-\bigl(p_d^{\circ L}(x)\bigr)^2}
\cdot
\frac{1+p_d^{\circ L}(x)}{1+p_d^{\circ L}(y)}
\leq
\frac{1-\bigl(p_d^{\circ L}(y)\bigr)^2}
     {1-\bigl(p_d^{\circ L}(x)\bigr)^2}.
\]
Combining this inequality with \eqref{eq:rL_ratio_bound} proves
\eqref{eq:iterated_ratio_bound}.
\end{proof}

\begin{proof}[Proof of Theorem~\ref{thm:deflated_ns}]
Let $M=U\operatorname{diag}(\sigma_1,\ldots,\sigma_q)V^{\top}$ with
$\sigma_1\geq\cdots\geq\sigma_q>0$ be a compact singular value
decomposition of $M$. Then $\operatorname{polar}(M)=UV^{\top}$. The nonzero
singular values of $X^{\mathrm{base}}=M/\norm{M}_F$ are $\sigma_i/\norm{M}_F$,
$i=1,\ldots,q$, and their smallest value is $\sigma_q/\norm{M}_F=\ell$.
Therefore, $\sigma(X^{\mathrm{base}})\subset[\ell,1]$.

By the definition of the deflated initialization
$X^{\mathrm{def}}_{\gamma}=\gamma^{-1}U_kV_k^{\top}+R_k/\norm{R_k}_F$,
its two terms have orthogonal row and column spaces, so its nonzero singular
values are $\gamma^{-1}$ (with multiplicity $k$) and
$\sigma_i/\norm{R_k}_F$ for $i=k+1,\ldots,q$. The smallest of the latter is
$\sigma_q/\norm{R_k}_F=(\norm{M}_F/\norm{R_k}_F)(\sigma_q/\norm{M}_F)=\alpha_k\ell$,
so the smallest nonzero singular value of $X^{\mathrm{def}}_{\gamma}$ is
$\min(\alpha_k\ell,\gamma^{-1})=\varphi_k$, and consequently
$\sigma(X^{\mathrm{def}}_{\gamma})\subset[\varphi_k,1]$: indeed
$\gamma^{-1}\leq1$ since $\gamma\geq1$, and
$\sigma_i/\norm{R_k}_F\leq\norm{R_k}_2/\norm{R_k}_F\leq1$ for $i>k$.
We record two endpoint inequalities for later use. Since $1\leq k<q$,
we have $q\geq2$, hence $\norm{M}_F>\sigma_q$ and $\ell<1$. Moreover,
$\norm{R_k}_F\geq\sigma_q$, with equality exactly when $k=q-1$, so
$\ell<\alpha_k\ell\leq1$; together with $\gamma^{-1}\in(0,1]$ this gives
$0<\varphi_k\leq1$.

If $X=U\operatorname{diag}(x_1,\ldots,x_q)V^{\top}$ with $x_i\in[0,1]$,
then $\mathcal{N}_d(X)=U\operatorname{diag}\bigl(p_d(x_1),\ldots,p_d(x_q)\bigr)V^{\top}$.
Thus, after $L$ iterations, a singular value with initial value $x_i$
becomes $p_d^{\circ L}(x_i)$.

Lemma~\ref{lem:one_step_ns_residual} shows that $p_d$ is nondecreasing
and maps $[0,1]$ into itself; the same is therefore true of
$p_d^{\circ L}$. Hence the largest error occurs in a direction
corresponding to the smallest initial singular value. If
$x_{\min}^{\star}$ denotes the smallest nonzero singular value of
$X^{\star}$, where $\star\in\{\mathrm{base},\mathrm{def}\}$, then
\begin{equation}
\bigl\|
\operatorname{polar}(M)
-\mathcal{N}_d^{\circ L}(X^{\star})
\bigr\|_2
=
1-p_d^{\circ L}(x_{\min}^{\star}).
\label{eq:matrix_error_scalarized}
\end{equation}
Indeed, the nonzero singular values of $\operatorname{polar}(M)$ are all one,
and $p_d^{\circ L}(x_i)\in[0,1]$.

For the deflated initialization,
$x_{\min}^{\mathrm{def}}=\varphi_k$. Since
$0<\varphi_k\leq1$, applying
\eqref{eq:iterated_absolute_bound} with $x=\varphi_k$ gives
\[
\bigl\|
\operatorname{polar}(M)
-\mathcal{N}_d^{\circ L}(X^{\mathrm{def}}_{\gamma})
\bigr\|_2
=
1-p_d^{\circ L}(\varphi_k)
\leq
\bigl(1-\varphi_k^2\bigr)^{(d+1)^L}.
\]

For the relative bound, note that
$x_{\min}^{\mathrm{base}}=\ell$ and
$x_{\min}^{\mathrm{def}}=\varphi_k$. The hypothesis $\gamma^{-1}>\ell$,
together with $\alpha_k\ell>\ell$, gives $\varphi_k>\ell$. Since
$0<\ell<\varphi_k\leq1$ and $\ell<1$, the hypotheses of
Lemma~\ref{lem:iterated_ns_residual} hold; applying
\eqref{eq:iterated_ratio_bound} with $x=\ell$ and
$y=\varphi_k$, together with \eqref{eq:matrix_error_scalarized},
gives
\[
\frac{
\bigl\|
\operatorname{polar}(M)
-\mathcal{N}_d^{\circ L}(X^{\mathrm{def}}_{\gamma})
\bigr\|_2
}{
\bigl\|
\operatorname{polar}(M)
-\mathcal{N}_d^{\circ L}(X^{\mathrm{base}})
\bigr\|_2
}
=
\frac{1-p_d^{\circ L}(\varphi_k)}{1-p_d^{\circ L}(\ell)}
\leq
\left(
\frac{1-\varphi_k^2}{1-\ell^2}
\right)^{(d+1)^L}.
\]
Finally, $\varphi_k>\ell$ gives $\varphi_k^2>\ell^2$, and
$\varphi_k\leq1$ gives $1-\varphi_k^2\geq0$. Since $\ell<1$, the
denominator $1-\ell^2$ is positive, and we conclude
$0\leq(1-\varphi_k^2)/(1-\ell^2)<1$, which proves the final strict inequality.
\end{proof}

%!TEX root = ../main.tex

\section{Supporting Figures for \S\ref{sec:deflation}}
\label{app:toy_deflation}

Figure~\ref{fig:spectrum} shows the momentum spectra measured during GPT-2
training that motivate \S\ref{sec:spectrum}, and
Figure~\ref{fig:toy_deflation} traces every singular value of the toy
spectrum of Example~\ref{ex:toy} through five steps of the classical
degree-five Newton--Schulz iteration \eqref{eq:muon_ns}, without and with
deflation of the head $\sigma_1$.

\begin{figure}[h]
\begin{center}
\includegraphics[width=0.8\linewidth]{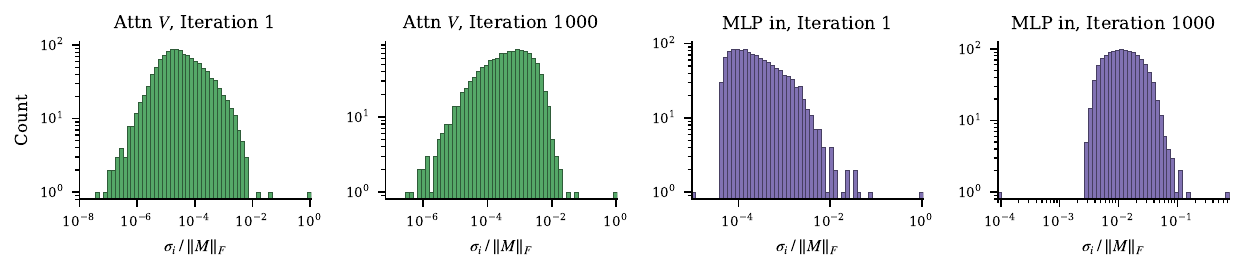}
\end{center}
\caption{Frobenius-normalized singular value spectra $\hat{\sigma}_i(M)$ of the
momentum matrices tracked by Muon in the middle transformer block (layer
$\lfloor N/2\rfloor$ of $N$ layers) of GPT-2 Large: the four attention
projections $\mathrm{Q}$, $\mathrm{K}$, $\mathrm{V}$, $\mathrm{O}$ and the two
MLP projections, at training steps $1$ and $1000$.}
\label{fig:spectrum}
\end{figure}

\begin{figure}[h]
\begin{center}
\includegraphics[width=0.85\linewidth]{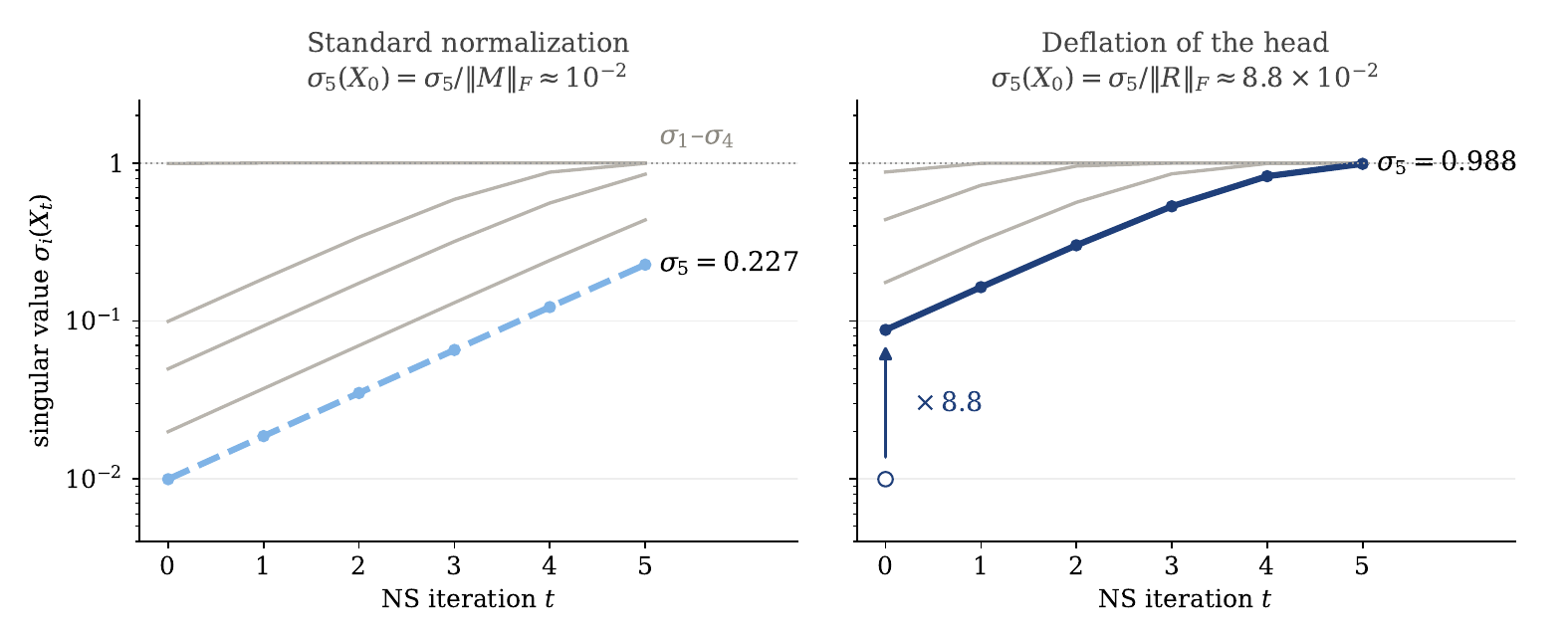}
\end{center}
\caption{Effect of deflation on the toy spectrum
$(\sigma_1,\ldots,\sigma_5)=(10,1,0.5,0.2,0.1)$ of Example~\ref{ex:toy} under
five steps of the classical degree-five Newton--Schulz iteration
\eqref{eq:muon_ns}.
Left: standard Frobenius normalization $X_0=M/\norm{M}_F$.
Right: deflation of the head followed by normalization of the residual.
The hollow marker shows where $\sigma_5$ would start without deflation, and
the arrow its lift by $\alpha_1=\norm{M}_F/\norm{R}_F\approx8.8$; the remaining
singular values (gray) are lifted by the same factor.}
\label{fig:toy_deflation}
\end{figure}

%!TEX root = ../main.tex

\section{Batched Randomized SVD: Implementation and Estimation Accuracy}
\label{app:brsvd}

This appendix states the batched randomized SVD used by
\S\ref{sec:deflated_muon} (Algorithm~\ref{alg:brsvd}), records the
implementation choices behind it, and quantifies how accurately the two
estimators of \S\ref{sec:deflation}---Algorithm~\ref{alg:brsvd} for Muon and
the warm-started tracking of Algorithm~\ref{alg:deflation_admm} for ADMM---recover
the components they deflate.

\begin{algorithm}[h]
\caption{\textsc{Batched Randomized SVD}}
\label{alg:brsvd}
\begin{algorithmic}[1]
\Require $A\in\R^{B\times n\times m}$, $n\ge m$; window $w\in(0,1)$; oversampling ratio $\delta\ge0$; subspace iterations $J\ge0$
\Ensure $U\in\R^{B\times n\times r}$, $s\in\R^{B\times r}$, $V\in\R^{B\times m\times r}$, with $r=\ceil{wm}$
\State $r\gets\ceil{wm}$;\quad $r'\gets\min\{m,\ r+\ceil{\delta m}\}$
\State $\Omega\gets\mathrm{Gaussian}(B,m,r')$
\State $Q\gets\mathrm{CholeskyQR}(A\Omega)$
\For{$j=1,\ldots,J$}
   \State $Y\gets\mathrm{CholeskyQR}(A^{\top}Q)$;\quad $Q\gets\mathrm{CholeskyQR}(AY)$
\EndFor
\State $G\gets Q^{\top}A$
\State $(U_G,s,V)\gets\mathrm{BatchedThinSVD}(G)$ \Comment{$G=U_G\operatorname{diag}(s)V^{\top}$, $s_1\ge\cdots\ge s_{r'}\ge0$}
\State $U\gets QU_G$
\State $U\gets U_{:,\,:,\,1:r}$;\quad $s\gets s_{:,\,1:r}$;\quad $V\gets V_{:,\,:,\,1:r}$ \Comment{drop the oversampling columns}
\State \Return $(U,s,V)$
\end{algorithmic}
\end{algorithm}

\paragraph{Grouping.}
Momentum matrices are grouped by shape: all matrices with the same $(n,m)$
are stacked into one batch, and Algorithm~\ref{alg:brsvd} runs once per
group per optimizer step. Wide matrices ($n<m$) are transposed before
grouping and the resulting polar factor is transposed back, using
$\operatorname{polar}(M^{\top})=\operatorname{polar}(M)^{\top}$.
For GPT-2 Large this yields two groups per optimizer step: the $1280\times1280$
group holding the attention $\mathrm{Q}$, $\mathrm{K}$, $\mathrm{V}$, and
$\mathrm{O}$ projections of all $36$ blocks ($B=144$), and the
$5120\times1280$ group holding the two MLP projections ($B=72$).

\paragraph{Operating point.}
The experiments use $w=0.025$ and $\delta=0.025$, so the sketch has
$r'=2r=\ceil{0.05\,m}$ columns, with $J=1$ subspace iteration.

\paragraph{Symmetric warm-started variant.}
The deflated PSD projection (\S\ref{sec:deflated_psd}) uses a symmetric
variant of Algorithm~\ref{alg:brsvd} in the spirit of warm-started block
eigensolvers~\citep{knyazev2001lobpcg}. For a symmetric input
$Z\in\Sn$, the routine tracks a single orthonormal basis $V\in\R^{n\times K}$
with $K=\ceil{w'n}$, returns Ritz values and vectors from the
Rayleigh--Ritz step, and warm-starts the subspace from the previous ADMM
iterate; since consecutive iterates are close, a single subspace-iteration
pass ($J=1$) suffices per ADMM iteration.

\paragraph{How inexact components enter the bounds.}
Let $(\widehat U_k,\hat s,\widehat V_k)$ be the estimated triplets that Algorithm~\ref{alg:deflated_muon} actually uses, $\widehat H:=\widehat U_k\operatorname{diag}(\hat s)\widehat V_k^{\top}$ and $\widehat R:=M-\widehat H$, and let
\[
\widehat X_0:=\widehat R/\|\widehat R\|_F+\gamma^{-1}\widehat U_k\widehat V_k^{\top}
\]
be the resulting entry; $X^{\mathrm{def}}_{\gamma}$ of \eqref{eq:deflated_init} is the same entry built from the exact triplets with the same $k$.
The iteration $\mathcal{N}_d^{\circ L}$ acts on the singular values of its argument and leaves its singular vectors unchanged, so the output $\mathcal{N}_d^{\circ L}(\widehat X_0)$ has singular values $p_d^{\circ L}\bigl(\sigma_i(\widehat X_0)\bigr)$ in the singular directions of $\widehat X_0$.
By Weyl's inequality, $|\sigma_i(\widehat X_0)-\sigma_i(X^{\mathrm{def}}_{\gamma})|\leq\eta:=\|\widehat X_0-X^{\mathrm{def}}_{\gamma}\|_2$, hence $\sigma(\widehat X_0)\subset[\varphi_k-\eta,\,1+\eta]$: the spectral floor of Theorem~\ref{thm:deflated_ns} degrades from $\varphi_k$ to $\varphi_k-\eta$, and since $p_d$ is nondecreasing with $p_d(1+t)-1=O(t^{d+1})$ (Lemma~\ref{lem:one_step_ns_residual}), the part of the spectrum above one is contracted at the same order; in particular, the orthogonality defect $\|X_L^{\top}X_L-I\|$ of the output depends on $\sigma(\widehat X_0)$ alone.
The limit of the iteration is $\operatorname{polar}(\widehat X_0)$ rather than $\operatorname{polar}(M)=\operatorname{polar}(X^{\mathrm{def}}_{\gamma})$, and the perturbation theory of the polar factor~\citep{li1995new,higham2008functions} bounds the difference by $2\|\widehat X_0-X^{\mathrm{def}}_{\gamma}\|_F/(2\varphi_k-\eta)$.
Both effects are governed by $\eta$, and expanding
\[
\widehat X_0-X^{\mathrm{def}}_{\gamma}
=
\frac{H_k-\widehat H}{\|\widehat R\|_F}
+R_k\Bigl(\frac{1}{\|\widehat R\|_F}-\frac{1}{\|R_k\|_F}\Bigr)
+\gamma^{-1}\bigl(\widehat U_k\widehat V_k^{\top}-U_kV_k^{\top}\bigr)
\]
shows that $\eta$ is small whenever the head singular values and the two head subspaces are accurate; the singular values are compared with an exact SVD below, and the same holds for Algorithm~\ref{alg:deflation_admm} with eigenpairs in place of singular triplets.

\paragraph{Bounds for randomized subspace iteration.}
Algorithm~\ref{alg:brsvd} is randomized subspace iteration with $J$ power steps and a Gaussian sketch of $r'$ columns~\citep{halko2011finding}.
Write $\Omega_1=V_k^{\top}\Omega$ and $\Omega_2=V_{\perp}^{\top}\Omega$ for the components of the sketch along and orthogonal to the exact right head subspace, and $p=r'-k$ for the oversampling.
\citet{saibaba2019randomized} bounds the canonical angles between the estimated and the exact left head subspaces by $\sin\theta_i\leq(\sigma_{k+1}/\sigma_i)^{2J+1}\|\Omega_2\Omega_1^{\dagger}\|_2$ for $i\leq k$, with an analogous bound for the right subspace, and \citet{gu2015subspace} shows $\sigma_i\geq\hat s_i\geq\sigma_i\big/\sqrt{1+(\sigma_{k+1}/\sigma_i)^{2(2J+1)}\|\Omega_2\Omega_1^{\dagger}\|_2^2}$ for the singular values; for a Gaussian sketch, $\mathbb{E}\|\Omega_2\Omega_1^{\dagger}\|_2\leq\sqrt{k/(p-1)}+e\sqrt{(k+p)(m-k)}/p$~\citep{halko2011finding,saibaba2019randomized}.
The error decays with the gap ratio $\sigma_{k+1}/\sigma_i$ raised to the power $2J+1$, so the single subspace iteration of the training runs ($J=1$) already suppresses it.
The bounds depend on the gap at the cutoff rather than on the gate threshold $\tau$: the gate certifies $\sigma_{k+1}<\tau\sigma_1<\sigma_k$ but not a gap between $\sigma_k$ and $\sigma_{k+1}$, so when the spectrum is flat around $\tau\sigma_1$ the trailing head directions are resolved less accurately; such directions carry singular values close to $\tau\sigma_1$ on both sides of the cutoff, and exchanging them changes $\widehat X_0$ little.

\paragraph{Accuracy of the estimated head.}
On the $24$ momentum matrices of \S\ref{subsec:convergence_experiments} we run Algorithm~\ref{alg:brsvd} with the constants of the training runs ($w=\delta=0.025$, $J=1$, single precision) for $20$ Gaussian sketches each and compare the $k$ deflated singular values---$k$ being the width chosen by the gate---with the leading $k$ singular values of a double-precision SVD; the gate fires on $21$ of the $24$ matrices.
Table~\ref{tab:rsvd_accuracy} summarizes the relative error $\max_{i\le k}\abs{\hat s_i-\sigma_i}/\sigma_i$ for each momentum matrix over the training steps and sketches.

\begin{table}[h]
\centering
\footnotesize
\setlength{\tabcolsep}{5pt}
\begin{tabular}{lccc}
\toprule
Momentum matrix & median & min & max \\
\midrule
Attention $\mathrm{Q}$ & 8.0e-3 & 2.3e-8 & 5.6e-2 \\
Attention $\mathrm{K}$ & 1.3e-2 & 2.9e-7 & 6.9e-2 \\
Attention $\mathrm{V}$ & 8.8e-8 & 8.1e-9 & 9.0e-7 \\
Attention $\mathrm{O}$ & 5.8e-7 & 1.1e-8 & 2.8e-2 \\
MLP (first) & 6.7e-3 & 3.4e-8 & 4.8e-2 \\
MLP (second) & 5.1e-4 & 1.7e-8 & 2.4e-3 \\
\bottomrule
\end{tabular}
\caption{Relative error $\max_{i\le k}\abs{\hat s_i-\sigma_i}/\sigma_i$ of the deflated singular values estimated by the batched randomized SVD on the momentum matrices of \S\ref{subsec:convergence_experiments}, at the operating point of the training runs; statistics over the training steps at which the gate fires and over $20$ sketches per matrix.}
\label{tab:rsvd_accuracy}
\end{table}

\paragraph{Tracking accuracy along the ADMM trajectory.}
We rerun the deflated ADMM of Figure~\ref{fig:admm_convergence} on the four instances of \S\ref{sec:sdp_experiments} and, every $100$ iterations, eigendecompose the current $Z$ exactly in double precision.
Let $\lambda_i$ be the Ritz values of Algorithm~\ref{alg:deflation_admm}, $\mathcal{I}=\{i:\abs{\lambda_i}>\tau s_1\}$ the deflated set, and $\lambda^{\star}_i$ the exact eigenvalues, matched by sign and magnitude order.
Figure~\ref{fig:tracking_accuracy} reports $\max_{i\in\mathcal I}\abs{\lambda_i-\lambda^{\star}_i}/\abs{\lambda^{\star}_i}$ at the checkpoints where the gate is active: the error stays below $3\times10^{-7}$.

\begin{figure}[h]
\begin{center}
\includegraphics[width=\linewidth]{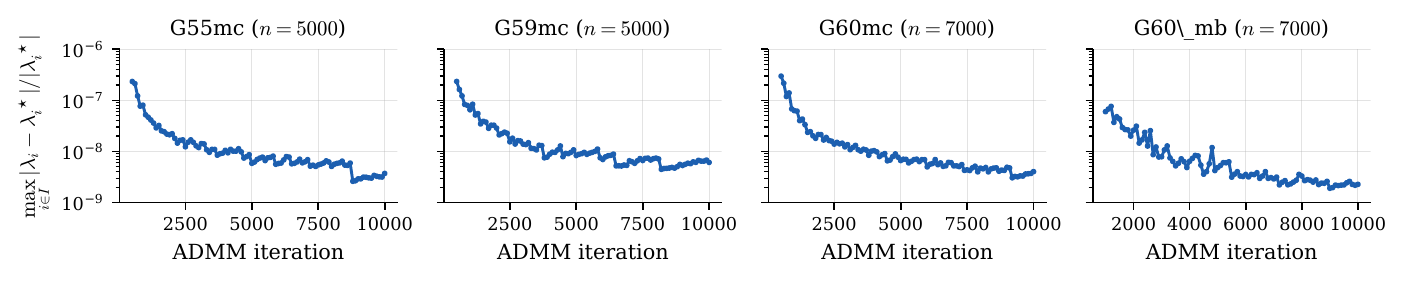}
\end{center}
\caption{Relative error of the deflated Ritz values inside deflated ADMM on the instances of \S\ref{sec:sdp_experiments}, against an exact double-precision eigendecomposition of the current $Z$ every $100$ iterations, at the checkpoints where the gate is active.}
\label{fig:tracking_accuracy}
\end{figure}

%!TEX root = ../main.tex

\section{Deflated PSD Projection: Details and Analysis}
\label{app:deflated_admm}

This appendix collects the material behind \S\ref{sec:deflated_psd}: the full
ADMM iteration (\S\ref{app:admm_sdp}), the construction of the fixed composite
filter (\S\ref{app:filter}), the proofs and the idealized analysis
(\S\ref{app:deflated_admm_error}), and the experimental details
(\S\ref{app:admm_setup}).

\subsection{ADMM for Standard-Form SDP}
\label{app:admm_sdp}

We solve the standard-form pair \eqref{eq:sdp_pair} with a classical
three-block ADMM~\citep{wen2010admm}; with penalty parameter $\sigma>0$, one
iteration is
\begin{equation}
\begin{aligned}
y^{(k+1)}
&=
(\mathcal{A}\mathcal{A}^{*})^{-1}
\bigl(
\sigma^{-1}b
-\mathcal{A}(\sigma^{-1}X^{(k)}+S^{(k)}-C)
\bigr),\\
S^{(k+1)}
&=
\Pi_{\psd{n}}\bigl(Z^{(k+1)}\bigr),
\qquad
Z^{(k+1)}
:=
C-\mathcal{A}^{*}y^{(k+1)}-\sigma^{-1}X^{(k)},\\
X^{(k+1)}
&=
X^{(k)}
+\sigma\bigl(S^{(k+1)}+\mathcal{A}^{*}y^{(k+1)}-C\bigr).
\end{aligned}
\label{eq:admm_iter}
\end{equation}
The $y$-update is a linear solve with the fixed matrix
$\mathcal{A}\mathcal{A}^{*}$ and the $X$-update is a matrix addition, so the
per-iteration cost is dominated by the projection in the $S$-update, which we
evaluate by Algorithm~\ref{alg:deflation_admm}.

\paragraph{Normalization scale.}
To place the spectrum of $X_0$ in $[-1,1]$, the scale $\theta$ (and
$\hat{\theta}$ for the residual) is obtained from a 20-step Lanczos
run~\citep[Chapter~13]{parlett1998symmetric} on the squared matrix: if
$(\rho,v)$ is the largest Ritz pair of $Z^2$, we use
$\theta=\sqrt{\rho+\norm{Z^2v-\rho v}_2}$.
This value is a certified upper bound on $\norm{Z}_2$ when $\rho$ is the Ritz
value closest to $\lambda_{\max}(Z^2)$~\citep{kang2025psdprojection}; taking
the \emph{largest} Ritz value of a random-start run is a heuristic, which
\citet{kang2025psdprojection} found to yield a valid upper bound in all their
experiments.

\subsection{Composite Minimax Filter}
\label{app:filter}

The component polynomials of the fixed filter $g$ are constructed offline in
two stages, following \citet{kang2025psdprojection}, whose construction adapts
the composite minimax sign approximation of \citet{lee2022minimax}. Let
$\mathcal{P}_{d_t}^{\mathrm{odd}}$ denote the space of odd real
polynomials of degree at most $d_t$, and let $\varepsilon>0$ exclude a
small neighborhood of the discontinuity of the sign function. The first
stage constructs a minimax approximation to $\operatorname{sign}(x)$
through sequential Remez problems. Starting from
$[a_1,b_1]=[\varepsilon,1]$, the polynomial at stage $t$ is chosen as
\begin{equation}
g_t
=
\operatorname*{arg\,min}_{g\in\mathcal{P}_{d_t}^{\mathrm{odd}}}
\max_{x\in[a_t,b_t]}\abs{g(x)-1},
\qquad
[a_{t+1},b_{t+1}]
=
g_t([a_t,b_t]).
\label{eq:admm_remez_construction}
\end{equation}
Because $g_t$ is odd, it simultaneously approximates $-1$ on the
reflected negative interval. Repeating this construction yields a
minimax-optimal composite approximation
$g_T\circ\cdots\circ g_1$ to the sign function on
$[-1,-\varepsilon]\cup[\varepsilon,1]$.

Since a composition optimized for the sign function is not necessarily
optimal for the matrix ReLU, the second stage initializes the
component polynomials with $\{g_t\}_{t=1}^T$ and jointly refines their
coefficients by approximately minimizing
\begin{equation}
\min_{p_1,\ldots,p_T}
\max_{x\in[-1,1]}
\left|
\frac{x}{2}
\left(1+(p_T\circ\cdots\circ p_1)(x)\right)
-\max(x,0)
\right|.
\label{eq:admm_relu_refinement}
\end{equation}
The maximum is evaluated over a dense set of points, and the polynomial
coefficients are refined by gradient descent while retaining the
composite representation. This avoids expanding the composition into a
single high-degree polynomial and preserves efficient matrix evaluation
through GEMMs.

For all experiments, we use the resulting filter in half precision, with
$\varepsilon=10^{-3}$ and $T=7$ degree-five odd polynomials of the form
$p_t(x)=\alpha_{t,1}x+\alpha_{t,3}x^3+\alpha_{t,5}x^5$.
Their coefficients remain fixed throughout ADMM. Deflation
changes only the normalized input to this filter and the subsequent
restoration of the deflated eigenspaces.

\subsection{Error Analysis}
\label{app:deflated_admm_error}

We first prove Proposition~\ref{prop:deflated_psd_ns}, restated for
self-containment together with the bound on the deflated error itself that the
proof yields alongside the ratio bound. The proof follows the three steps of
Appendix~\ref{app:proof_deflated_ns}: it identifies the spectral floors of the
two initializations, reduces the matrix error to a scalar residual by a
spectral computation, and applies
Lemmas~\ref{lem:one_step_ns_residual}--\ref{lem:iterated_ns_residual} at the
floors $\ell$ and $\alpha_k\ell$.

\begin{proprestated}
Let $Z\in\Sn$ be nonzero with nonzero eigenvalues $\lambda_1,\ldots,\lambda_q$
ordered so that $\sigma_i=\abs{\lambda_i}$ satisfies
$\sigma_1\geq\cdots\geq\sigma_q>0$, let $1\leq k<q$ with $\sigma_{k+1}<\sigma_1$,
let $\gamma\geq1$, and let
$X^{\mathrm{base}}=Z/\norm{Z}_2$,
$X^{\mathrm{def}}_{\gamma}=\gamma^{-1}U_k\operatorname{sign}(\Lambda_k)U_k^{\top}+R_k/\norm{R_k}_2$,
$\ell=\sigma_q/\norm{Z}_2$, and $\alpha_k=\norm{Z}_2/\norm{R_k}_2$ be as in
\S\ref{subsec:error_bounds}, eq.~\eqref{eq:deflated_init} with $A=Z$, and let
$\widehat{\Pi}^{\mathrm{base}}(Z)=\half Z\bigl(I+\mathcal{N}_d^{\circ L}(X^{\mathrm{base}})\bigr)$ and
$\widehat{\Pi}^{\mathrm{def}}_{\gamma}(Z)=U_k(\Lambda_k)_{+}U_k^{\top}+\half R_k\bigl(I+\mathcal{N}_d^{\circ L}(X^{\mathrm{def}}_{\gamma})\bigr)$.
Then $\sigma(X^{\mathrm{base}})\subset[\ell,1]$ and
$\sigma(X^{\mathrm{def}}_{\gamma})\subset\{\gamma^{-1}\}\cup[\alpha_k\ell,1]$,
and $\widehat{\Pi}^{\mathrm{def}}_{\gamma}(Z)$ does not depend on $\gamma$.
Moreover, for every $d\geq1$ and $L\geq0$,
\begin{equation}
\bigl\|
\widehat{\Pi}^{\mathrm{def}}_{\gamma}(Z)-\Pi_{\psd{n}}(Z)
\bigr\|_2
\leq
\frac{\norm{R_k}_2}{2}
\max_{\alpha_k\ell\leq x\leq1}x\bigl(1-x^2\bigr)^{(d+1)^L}
\leq
\frac{\norm{R_k}_2}{2}
\bigl(1-\alpha_k^2\ell^2\bigr)^{(d+1)^L},
\label{eq:deflated_psd_absolute}
\end{equation}
and the deflated-to-baseline error ratio satisfies
\[
\frac{
\bigl\|
\widehat{\Pi}^{\mathrm{def}}_{\gamma}(Z)-\Pi_{\psd{n}}(Z)
\bigr\|_2
}{
\bigl\|
\widehat{\Pi}^{\mathrm{base}}(Z)-\Pi_{\psd{n}}(Z)
\bigr\|_2
}
\leq
\left(
\frac{1-\alpha_k^2\ell^2}{1-\ell^2}
\right)^{(d+1)^L}
<1.
\]
\end{proprestated}

\begin{proof}[Proof of Proposition~\ref{prop:deflated_psd_ns}]
Write the spectral decomposition $Z=\sum_{i=1}^{q}\lambda_iu_iu_i^{\top}$ with
orthonormal $u_1,\ldots,u_q$ and $\sigma_i=\abs{\lambda_i}$, so that
$U_k\Lambda_kU_k^{\top}=\sum_{i\leq k}\lambda_iu_iu_i^{\top}$,
$R_k=\sum_{i>k}\lambda_iu_iu_i^{\top}$, and
$\Pi_{\psd{n}}(Z)=\sum_{i=1}^{q}(\lambda_i)_+u_iu_i^{\top}$.
Set $\theta:=\theta(Z)=\norm{Z}_2=\abs{\lambda_1}$ and
$\theta_k:=\theta(R_k)=\norm{R_k}_2=\abs{\lambda_{k+1}}$, so that $\alpha_k=\theta/\theta_k$.

\emph{Step 1: spectral floors.}
The nonzero eigenvalues of $X^{\mathrm{base}}=Z/\theta$ are $\lambda_i/\theta$,
$i=1,\ldots,q$, with magnitudes in $[\abs{\lambda_q}/\theta,1]=[\ell,1]$; hence
$\sigma(X^{\mathrm{base}})\subset[\ell,1]$.
The two terms of $X^{\mathrm{def}}_{\gamma}$ act on orthogonal eigenspaces, so its
nonzero eigenvalues are $\gamma^{-1}\operatorname{sign}(\lambda_i)$ for $i\leq k$
and $\lambda_i/\theta_k$ for $i>k$.
The magnitudes of the latter lie in $[\abs{\lambda_q}/\theta_k,1]$, and
$\abs{\lambda_q}/\theta_k=(\theta/\theta_k)(\abs{\lambda_q}/\theta)=\alpha_k\ell$;
hence $\sigma(X^{\mathrm{def}}_{\gamma})\subset\{\gamma^{-1}\}\cup[\alpha_k\ell,1]$.
We record two endpoint inequalities for later use: the hypothesis
$\abs{\lambda_{k+1}}<\abs{\lambda_1}$ gives $\alpha_k>1$ and
$\ell\leq\abs{\lambda_{k+1}}/\abs{\lambda_1}<1$, and
$\abs{\lambda_q}\leq\abs{\lambda_{k+1}}$ gives $\alpha_k\ell\leq1$; thus
$0<\ell<\alpha_k\ell\leq1$.

\emph{Step 2: scalarization.}
For a symmetric $X$ we have $XX^{\top}=X^2$, so
\eqref{eq:general_ns_iteration} reads
$\mathcal{N}_d(X)=\sum_{i=0}^{d}a_iX^{2i+1}=p_d(X)$, with $p_d$ the odd
scalar polynomial of Lemma~\ref{lem:one_step_ns_residual}.
Consequently, if $X=\sum_ix_iu_iu_i^{\top}$, then
$\mathcal{N}_d^{\circ L}(X)=\sum_ip_d^{\circ L}(x_i)u_iu_i^{\top}$, and
$p_d^{\circ L}(-x)=-p_d^{\circ L}(x)$.
Since $R_ku_i=0$ for $i\leq k$, the padded head of $X^{\mathrm{def}}_{\gamma}$
is annihilated in the product
\[
R_k\,\mathcal{N}_d^{\circ L}\bigl(X^{\mathrm{def}}_{\gamma}\bigr)
=
\sum_{i>k}\lambda_i\,p_d^{\circ L}(\lambda_i/\theta_k)\,u_iu_i^{\top},
\]
which does not involve $\gamma$; this proves that
$\widehat{\Pi}^{\mathrm{def}}_{\gamma}(Z)$ is independent of $\gamma$.
Using $(\lambda)_+=\frac{\lambda}{2}\bigl(1+\operatorname{sign}(\lambda)\bigr)$
and the oddness of $p_d^{\circ L}$, we have
$\lambda\bigl(p_d^{\circ L}(\lambda/\vartheta)-\operatorname{sign}(\lambda)\bigr)
=-\abs{\lambda}\bigl(1-p_d^{\circ L}(\abs{\lambda}/\vartheta)\bigr)$
for every $\lambda\neq0$ and every scale $\vartheta>0$.
The head terms $U_k(\Lambda_k)_{+}U_k^{\top}$ and
$\sum_{i\leq k}(\lambda_i)_+u_iu_i^{\top}$ cancel, and the kernel of $Z$
contributes nothing to either approximation, so
\begin{align}
\widehat{\Pi}^{\mathrm{def}}_{\gamma}(Z)-\Pi_{\psd{n}}(Z)
&=
-\half\sum_{i>k}\abs{\lambda_i}
\bigl(1-p_d^{\circ L}(\abs{\lambda_i}/\theta_k)\bigr)u_iu_i^{\top},
\label{eq:psd_error_def}\\
\widehat{\Pi}^{\mathrm{base}}(Z)-\Pi_{\psd{n}}(Z)
&=
-\half\sum_{i=1}^{q}\abs{\lambda_i}
\bigl(1-p_d^{\circ L}(\abs{\lambda_i}/\theta)\bigr)u_iu_i^{\top}.
\label{eq:psd_error_base}
\end{align}
Let $x_i:=\abs{\lambda_i}/\theta\in[\ell,1]$ for $i=1,\ldots,q$, so that
$\abs{\lambda_i}/\theta_k=\alpha_kx_i\in[\alpha_k\ell,1]$ for $i>k$.
By Lemma~\ref{lem:one_step_ns_residual}, $p_d^{\circ L}$ maps $[0,1]$ into
itself, so every coefficient in
\eqref{eq:psd_error_def}--\eqref{eq:psd_error_base} is nonnegative; both error
matrices are negative semidefinite, and by the unitary invariance of the
spectral norm,
\begin{align}
\bigl\|\widehat{\Pi}^{\mathrm{def}}_{\gamma}(Z)-\Pi_{\psd{n}}(Z)\bigr\|_2
&=
\frac{\theta_k}{2}\max_{k<i\leq q}
\alpha_kx_i\bigl(1-p_d^{\circ L}(\alpha_kx_i)\bigr),
\label{eq:psd_error_scalarized}\\
\bigl\|\widehat{\Pi}^{\mathrm{base}}(Z)-\Pi_{\psd{n}}(Z)\bigr\|_2
&=
\frac{\theta}{2}\max_{1\leq i\leq q}
x_i\bigl(1-p_d^{\circ L}(x_i)\bigr).
\label{eq:psd_error_scalarized_base}
\end{align}
These are the counterparts of \eqref{eq:matrix_error_scalarized}; the weight
$x_i$ is new, because the projection error in the direction of $\lambda_i$ is
proportional to $\abs{\lambda_i}$.

\emph{Step 3: bounds at the floors.}
For $i>k$, apply \eqref{eq:iterated_absolute_bound} with
$x=\alpha_kx_i\in[\alpha_k\ell,1]$:
\[
\alpha_kx_i\bigl(1-p_d^{\circ L}(\alpha_kx_i)\bigr)
\leq
\alpha_kx_i\bigl(1-\alpha_k^2x_i^2\bigr)^{(d+1)^L}
\leq
\max_{\alpha_k\ell\leq x\leq1}x(1-x^2)^{(d+1)^L}
\leq
(1-\alpha_k^2\ell^2)^{(d+1)^L},
\]
where the last step uses $x\leq1$ and the monotonicity of $1-x^2$.
Inserting this into \eqref{eq:psd_error_scalarized} proves
\eqref{eq:deflated_psd_absolute}.

For the ratio bound, fix $i>k$. Then $0<x_i\leq\alpha_kx_i\leq1$ and
$x_i\leq\abs{\lambda_{k+1}}/\abs{\lambda_1}<1$, so
\eqref{eq:iterated_ratio_bound} applies with $x=x_i$ and $y=\alpha_kx_i$:
\[
\frac{1-p_d^{\circ L}(\alpha_kx_i)}{1-p_d^{\circ L}(x_i)}
\leq
\left(\frac{1-\alpha_k^2x_i^2}{1-x_i^2}\right)^{(d+1)^L}
\leq
\left(\frac{1-\alpha_k^2\ell^2}{1-\ell^2}\right)^{(d+1)^L}
=:\rho.
\]
The second inequality holds because
$\frac{1-\alpha^2x^2}{1-x^2}=1-(\alpha^2-1)\frac{x^2}{1-x^2}$ is decreasing in
$x\in(0,1)$ for $\alpha>1$ and $x_i\geq\ell$; both bases are nonnegative
because $\alpha_kx_i\leq1$, so the inequality survives raising to the power
$(d+1)^L$.
Multiplying by $\abs{\lambda_i}=\theta_k\alpha_kx_i=\theta x_i$ gives, for every
$i>k$,
\[
\theta_k\,\alpha_kx_i\bigl(1-p_d^{\circ L}(\alpha_kx_i)\bigr)
\leq
\rho\,\theta\,x_i\bigl(1-p_d^{\circ L}(x_i)\bigr)
\leq
\rho\,\theta\max_{1\leq j\leq q}x_j\bigl(1-p_d^{\circ L}(x_j)\bigr).
\]
Taking the maximum over $i>k$ and using
\eqref{eq:psd_error_scalarized}--\eqref{eq:psd_error_scalarized_base} proves
the inequality in \eqref{eq:deflated_psd_ratio}.
The denominator is positive: $0<x_q=\ell<1$, and the recursion
\eqref{eq:iterated_residual_recursion} in the proof of
Lemma~\ref{lem:iterated_ns_residual} shows inductively that
$1-\bigl(p_d^{\circ L}(\ell)\bigr)^2>0$, hence $1-p_d^{\circ L}(\ell)>0$.
Finally, $0<\ell<\alpha_k\ell\leq1$ gives
$0\leq(1-\alpha_k^2\ell^2)/(1-\ell^2)<1$, which is the strict inequality.
\end{proof}

The proof uses the exact scales $\norm{Z}_2$ and $\norm{R_k}_2$ only through
$\theta\geq\norm{Z}_2$, $\hat{\theta}\geq\norm{R_k}_2$, and $\hat{\theta}<\theta$;
hence Proposition~\ref{prop:deflated_psd_ns} remains valid for the Lanczos
upper estimates of Algorithm~\ref{alg:deflation_admm}, with
$\ell=\abs{\lambda_q}/\theta$, $\alpha_k=\theta/\hat{\theta}$,
$\alpha_k\ell=\abs{\lambda_q}/\hat{\theta}$, and $\norm{R_k}_2$ replaced by
$\hat{\theta}$ in the prefactor of \eqref{eq:deflated_psd_absolute}.

\paragraph{The fixed filter.}
The implemented filter is fixed offline on $[\varepsilon,1]$
(\S\ref{app:filter}); the following bound takes it as given.

\begin{proposition}[Deflated projection with a fixed filter]
\label{prop:fixed_filter}
Let $Z\in\Sn$ have nonzero eigenvalues $\lambda_1,\ldots,\lambda_q$ with
$\abs{\lambda_1}\geq\cdots\geq\abs{\lambda_q}>0$, let $1\leq k<q$, and let
$Z=U_k\Lambda_kU_k^{\top}+R_k$ be the exact head split of \eqref{eq:identities}.
Let $g$ be a fixed odd polynomial (possibly composite) with sign error
$\delta_\varepsilon=\max_{\varepsilon\leq x\leq1}\abs{g(x)-1}$ on
$[\varepsilon,1]$, and let $\hat{\theta}\geq\norm{R_k}_2$.
If $\abs{\lambda_q}\geq\varepsilon\hat{\theta}$, then
\[
\Bigl\|
U_k(\Lambda_k)_{+}U_k^{\top}+\half R_k\bigl(I+g(R_k/\hat{\theta})\bigr)-\Pi_{\psd{n}}(Z)
\Bigr\|_2
\;\leq\;
\frac{\hat{\theta}}{2}\,
\max_{\varepsilon\leq x\leq1}x\,\abs{g(x)-1}
\;\leq\;
\frac{\hat{\theta}}{2}\,\delta_\varepsilon.
\]
\end{proposition}

\begin{proof}[Proof of Proposition~\ref{prop:fixed_filter}]
Write the spectral decomposition
$Z=\sum_{i=1}^{q}\lambda_iu_iu_i^{\top}$, so that
$R_k=\sum_{i=k+1}^{q}\lambda_iu_iu_i^{\top}$ and
$U_k(\Lambda_k)_{+}U_k^{\top}=\sum_{i=1}^{k}(\lambda_i)_+u_iu_i^{\top}$, the $u_i$ being the columns of $U_k$.
Define $h(x):=\frac{x}{2}\bigl(1+g(x)\bigr)$, so that
$\half R_k\bigl(I+g(R_k/\hat{\theta})\bigr)=\hat{\theta}\,h(R_k/\hat{\theta})$
and $h(0)=0$: the head directions and the kernel of $Z$ contribute nothing to
the filtered term.
Using spectral functional calculus and
$(\lambda)_+=\frac{\lambda}{2}\bigl(1+\operatorname{sign}(\lambda)\bigr)$,
\[
U_k(\Lambda_k)_{+}U_k^{\top}+\hat{\theta}\,h(R_k/\hat{\theta})-\Pi_{\psd{n}}(Z)
=
\frac12
\sum_{i=k+1}^{q}
\lambda_i
\left[
g\bigl(\lambda_i/\hat{\theta}\bigr)
-\operatorname{sign}(\lambda_i)
\right]u_iu_i^{\top}.
\]
By the orthogonality of the eigenvectors and the unitary invariance of the
spectral norm, with $x_i:=\lambda_i/\hat{\theta}$,
\[
\Bigl\|
U_k(\Lambda_k)_{+}U_k^{\top}+\hat{\theta}\,h(R_k/\hat{\theta})-\Pi_{\psd{n}}(Z)
\Bigr\|_2
=
\frac{\hat{\theta}}{2}
\max_{k<i\leq q}
\abs{x_i}
\,\bigl|
g(x_i)
-\operatorname{sign}(x_i)
\bigr|.
\]
Since $\hat{\theta}\geq\norm{R_k}_2$ and
$\abs{\lambda_q}\geq\varepsilon\hat{\theta}$, every $x_i$ with $i>k$
satisfies $\varepsilon\leq\abs{x_i}\leq1$. Because $g$ is odd,
$\abs{x}\,\abs{g(x)-\operatorname{sign}(x)}$ equals $\abs{x}\,\abs{g(\abs{x})-1}$, so
the maximum is at most
$\max_{\varepsilon\leq x\leq1}x\abs{g(x)-1}\leq\delta_\varepsilon$.
\end{proof}

If the floor condition fails, an eigenvalue with $\abs{\lambda_i}<\varepsilon\hat{\theta}$
contributes at most
$\half\abs{\lambda_i}\bigl(1+\abs{g(x_i)}\bigr)
\leq\frac{\hat{\theta}}{2}\,\varepsilon\bigl(1+\max_{\abs{x}\leq\varepsilon}\abs{g(x)}\bigr)$
to the maximum above, which yields a two-regime bound.

We next analyze the idealized case in which the composite filter is
regenerated on the deflated interval, so that the sequential minimax
optimality of the construction \eqref{eq:admm_remez_construction} can be
invoked. This result explains what the fixed filter gives up, and it recovers
the doubly exponential rate of Theorem~\ref{thm:deflated_ns}.

\begin{theorem}[Idealized sequential-minimax filter]
\label{thm:deflated_psd_error}
Let $Z=U_k\Lambda_kU_k^{\top}+R_k$ and $k$ be as in Proposition~\ref{prop:fixed_filter},
let $\theta_k:=\norm{R_k}_2=\abs{\lambda_{k+1}}$, and let
$a_k:=\abs{\lambda_q}/\theta_k\in(0,1]$, so that the nonzero eigenvalues of
$R_k/\theta_k$ lie in $[-1,-a_k]\cup[a_k,1]$.
Let $g=g_T\circ\cdots\circ g_1$ be the composite filter generated by the
sequential minimax construction \eqref{eq:admm_remez_construction}
initialized on $[a_k,1]$, where each $g_t$ is an odd polynomial of degree at
most $2d+1$, with integers $d,T\geq1$.
Define $h_g(x):=\frac{x}{2}\bigl(1+g(x)\bigr)$.
Then the deflated PSD approximation
$\widehat{P}_{\mathrm{def}}:=U_k(\Lambda_k)_{+}U_k^{\top}+\theta_k h_g(R_k/\theta_k)$
satisfies
\begin{equation}
\bigl\|
\widehat{P}_{\mathrm{def}}-\Pi_{\psd{n}}(Z)
\bigr\|_2
\leq
\frac{\theta_k}{2}(1-a_k^2)^{(d+1)^T}.
\label{eq:deflated_psd_explicit_bound}
\end{equation}
\end{theorem}

\begin{proof}
The spectral computation in the proof of Proposition~\ref{prop:fixed_filter},
with $\hat{\theta}=\theta_k$, gives
\begin{equation}
\bigl\|
\widehat{P}_{\mathrm{def}}-\Pi_{\psd{n}}(Z)
\bigr\|_2
\leq
\frac{\theta_k}{2}
\max_{a_k\leq x\leq1}\abs{1-g(x)}.
\label{eq:deflated_psd_spectral_error}
\end{equation}
It remains to bound the scalar sign-approximation error.
Consider the degree-$(2d+1)$ Newton--Schulz polynomial
$p_{\mathrm{NS},d}(x):=
x\sum_{j=0}^{d}c_j(1-x^2)^j$, with
$c_j:=4^{-j}\binom{2j}{j}$, and its $T$-fold composition
$f_T:=p_{\mathrm{NS},d}^{\circ T}$.
By Lemma~\ref{lem:one_step_ns_residual},
$p_{\mathrm{NS},d}$ maps $[0,1]$ into itself and satisfies
$0\leq1-p_{\mathrm{NS},d}(x)^2\leq(1-x^2)^{d+1}$
for $x\in[0,1]$.
For $x_0=x\in[0,1]$ and
$x_{t+1}=p_{\mathrm{NS},d}(x_t)$, set $r_t:=1-x_t^2$.
The preceding inequality gives $r_{t+1}\leq r_t^{d+1}$,
and induction yields
$0\leq1-f_T(x)^2\leq(1-x^2)^{(d+1)^T}$.
Moreover, $f_T(x)\in[0,1]$, and hence
\[
\abs{1-f_T(x)}
=
\frac{1-f_T(x)^2}{1+f_T(x)}
\leq
(1-x^2)^{(d+1)^T}.
\]
The polynomial $f_T$ satisfies the same degree and composition
constraints as $g$. By Theorem~1 of \citet{kang2025psdprojection},
the sequential minimax construction makes $g$ optimal over this
class for uniform approximation of the sign function on
$[-1,-a_k]\cup[a_k,1]$.
Therefore,
\begin{equation}
\max_{a_k\leq x\leq1}\abs{1-g(x)}
\leq
\max_{a_k\leq x\leq1}\abs{1-f_T(x)}
\leq
\max_{a_k\leq x\leq1}(1-x^2)^{(d+1)^T}
=
(1-a_k^2)^{(d+1)^T}.
\label{eq:deflated_psd_sign_bound}
\end{equation}
Substituting \eqref{eq:deflated_psd_sign_bound} into
\eqref{eq:deflated_psd_spectral_error} proves
\eqref{eq:deflated_psd_explicit_bound}.
\end{proof}

\subsection{Experimental Details}
\label{app:admm_setup}

For the experiments of \S\ref{sec:sdp_experiments}, we monitor the maximum
KKT residual
\begin{equation}
\eta
=
\max\Bigl\{
\tfrac{\norm{\mathcal{A}X-b}_2}{1+\norm{b}_2},\,
\tfrac{\norm{\mathcal{A}^{*}y+S-C}_F}{1+\norm{C}_F},\,
\tfrac{\abs{\inprod{C}{X}-b^{\top}y}}
{1+\abs{\inprod{C}{X}}+\abs{b^{\top}y}},\,
\tfrac{\max\{0,-\lambda_{\min}(X)\}}{1+\norm{b}_2},\,
\tfrac{\max\{0,-\lambda_{\min}(S)\}}{1+\norm{C}_F}
\Bigr\}.
\label{eq:admm_eta}
\end{equation}
Because computing the last two terms requires explicit eigenvalue
information, they are not evaluated during the ADMM iterations. The
convergence curves therefore report the maximum of the first three,
eigenvalue-free terms of \eqref{eq:admm_eta}. At the end of each run, the two
PSD-feasibility terms are additionally evaluated using a double-precision
eigendecomposition of the final iterates. In our experiments, the first three
terms dominate the final two terms, so the eigenvalue-free residual accurately
reflects the observed convergence behavior.

%!TEX root = ../main.tex

\section{Additional Experiments}
\label{app:additional}

\paragraph{Deflation cutoff.}
Table~\ref{tab:cutoff} reports the final validation loss when deflation is
enabled for different percentages of the optimizer steps
(\S\ref{sec:ablations}).

\begin{table}[h]
\centering
\footnotesize
\setlength{\tabcolsep}{6pt}
\begin{tabular}{lccccc}
\toprule
 & & \multicolumn{4}{c}{\textsc{Deflated Muon}, deflation cutoff} \\
\cmidrule(l){3-6}
Mapping & Muon & $10\%$ & $30\%$ & $50\%$ & $100\%$ \\
\midrule
Newton--Schulz & $3.3813$ & $3.3716$ & $3.3604$ & $3.3599$ & $\mathbf{3.3575}$ \\
Polar Express  & $3.3675$ & $3.3587$ & $3.3574$ & $\mathbf{3.3552}$ & $3.3579$ \\
\bottomrule
\end{tabular}
\caption{Final validation loss of GPT-2 Large when deflation is enabled for different percentages of the optimizer steps.
The \emph{Muon} column is the baseline without deflation; bold marks the best cutoff for each mapping, and every tested cutoff improves on the baseline.}
\label{tab:cutoff}
\end{table}

\paragraph{Deflation ratio.}
Tables~\ref{tab:gate_w} and~\ref{tab:gate_tau} report the sweeps over the window
ratio $w$ and the threshold $\tau$ of \S\ref{sec:ablations}: the final validation
loss and the number $k$ of singular pairs deflated per matrix, averaged over the
optimizer steps at which the gate fires.

\begin{table}[H]
\begin{minipage}[t]{0.49\linewidth}
\centering
\scriptsize
\setlength{\tabcolsep}{2.8pt}
\begin{tabular}{lccccc}
\toprule
 & \multicolumn{4}{c}{Window ratio $w$ ($\tau=0.1$)} & \\
\cmidrule(lr){2-5}
 & $0.0025$ & $0.005$ & $0.01$ & $0.025$ & Muon \\
\midrule
\multicolumn{6}{l}{\emph{Newton--Schulz}} \\
Val.\ loss   & $\mathbf{3.3574}$ & $\mathbf{3.3597}$ & $\mathbf{3.3568}$ & $\mathbf{3.3574}$ & $3.3794$ \\
Removed $k$  & $1.5$ & $2.3$ & $3.8$ & $8.8$ & -- \\
\midrule
\multicolumn{6}{l}{\emph{Polar Express}} \\
Val.\ loss   & $\mathbf{3.3546}$ & $\mathbf{3.3573}$ & $\mathbf{3.3542}$ & $\mathbf{3.3581}$ & $3.3671$ \\
Removed $k$  & $1.6$ & $2.5$ & $4.0$ & $9.4$ & -- \\
\bottomrule
\end{tabular}
\captionof{table}{Window ratio $w$ at $\tau=0.1$: final validation loss of GPT-2 Large under \textsc{Deflated Muon} and the average number $k$ of singular pairs deflated per matrix, over the optimizer steps at which the gate fires. The column ``Muon'' shows the baseline without deflation.}
\label{tab:gate_w}
\end{minipage}\hfill
\begin{minipage}[t]{0.49\linewidth}
\centering
\scriptsize
\setlength{\tabcolsep}{2.8pt}
\begin{tabular}{lccccc}
\toprule
 & \multicolumn{4}{c}{Threshold $\tau$ ($w=0.025$)} & \\
\cmidrule(lr){2-5}
 & $0.05$ & $0.1$ & $0.2$ & $0.3$ & Muon \\
\midrule
\multicolumn{6}{l}{\emph{Newton--Schulz}} \\
Val.\ loss   & $\mathbf{3.3581}$ & $\mathbf{3.3574}$ & $\mathbf{3.3575}$ & $\mathbf{3.3565}$ & $3.3794$ \\
Removed $k$  & $8.8$ & $8.8$ & $5.2$ & $3.3$ & -- \\
\midrule
\multicolumn{6}{l}{\emph{Polar Express}} \\
Val.\ loss   & $\mathbf{3.3605}$ & $\mathbf{3.3581}$ & $\mathbf{3.3596}$ & $\mathbf{3.3552}$ & $3.3671$ \\
Removed $k$  & $9.6$ & $9.4$ & $5.8$ & $4.0$ & -- \\
\bottomrule
\end{tabular}
\captionof{table}{Threshold $\tau$ at $w=0.025$, same quantities as Table~\ref{tab:gate_w}. Deflation consistently improves the validation loss over the Muon baseline.}
\label{tab:gate_tau}
\end{minipage}
\end{table}

\paragraph{Matrix-function accuracy on all recorded matrices.}
Figures~\ref{fig:convergence_mlp_all} and~\ref{fig:convergence_attn_all}
extend Figure~\ref{fig:convergence_attn_v} to all six momentum matrices
recorded in the middle block of GPT-2 Large: the MLP input and output
projections and the attention projections $\mathrm{Q}$, $\mathrm{K}$,
$\mathrm{V}$, $\mathrm{O}$, at training steps $1$, $500$, $1000$, and $1500$.
The deflated arms run the batched randomized SVD of Algorithm~\ref{alg:brsvd}
(a sketch of $\ceil{0.025m}$ columns plus $\ceil{0.025m}$ oversampling
columns, one subspace iteration) and the gate with $w=0.025$, $\tau=0.1$,
$\gamma=1.01$ for Newton--Schulz and $\gamma=1.1$ for Polar Express, exactly
as in training; the sketch is seeded per matrix.
Wherever the momentum has a pronounced head, deflation lowers the error at
every iteration for both mappings; on $\mathrm{K}$ at step $1000$ and on the
MLP output projection at steps $1000$ and $1500$ the spectrum has flattened,
the gate does not fire, and the deflated curves coincide with the baselines.
On $\mathrm{K}$ at step $1500$ the exact spectrum sits just above the gate
threshold ($\sigma_{32}/\sigma_1=0.105$ for the window $r=32$), but the
estimated $\sigma_{32}$ lies $5$--$7\%$ below the exact value for every
sketch seed we tried, so the production gate fires and deflates $31$
directions; the deflated curves then improve on the baselines by up to
$0.053$ (Newton--Schulz) and $0.042$ (Polar Express) over the five
iterations.
Replacing the estimated head by the exact one changes no other gate decision
and moves every other curve by at most $0.003$.
With Newton--Schulz the deflated iterate is never less accurate than the
baseline at any iteration; with Polar Express, the curves on the MLP
projections at later steps cross at intermediate iterations and agree to
within $0.003$ after $L=5$ iterations.
Evaluating the polynomial iterations in \texttt{bfloat16}, the training
precision, instead of single precision moves individual curves by at most
$0.017$; the ordering of the arms is unchanged, except that rounding lets the
Newton--Schulz deflated iterate exceed its baseline by up to $0.001$ at one
point (MLP input projection, step $1000$) and widens the Polar Express
agreement on the MLP projections to $0.004$.

\begin{figure}[!t]
\centering
\begin{subfigure}[t]{\linewidth}
\centering
\includegraphics[width=0.94\linewidth]{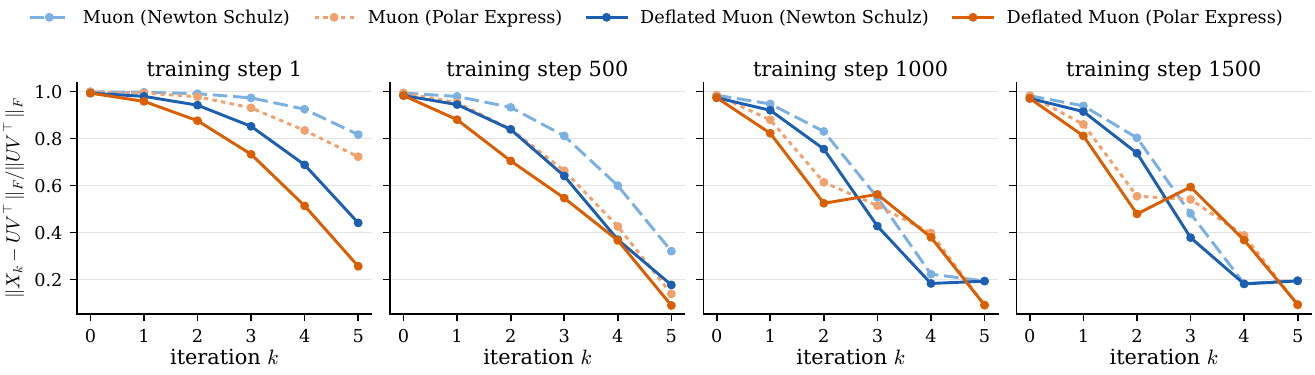}
\caption{MLP input projection.}
\label{fig:convergence_mlp_fc}
\end{subfigure}\\[1.5ex]
\begin{subfigure}[t]{\linewidth}
\centering
\includegraphics[width=0.94\linewidth]{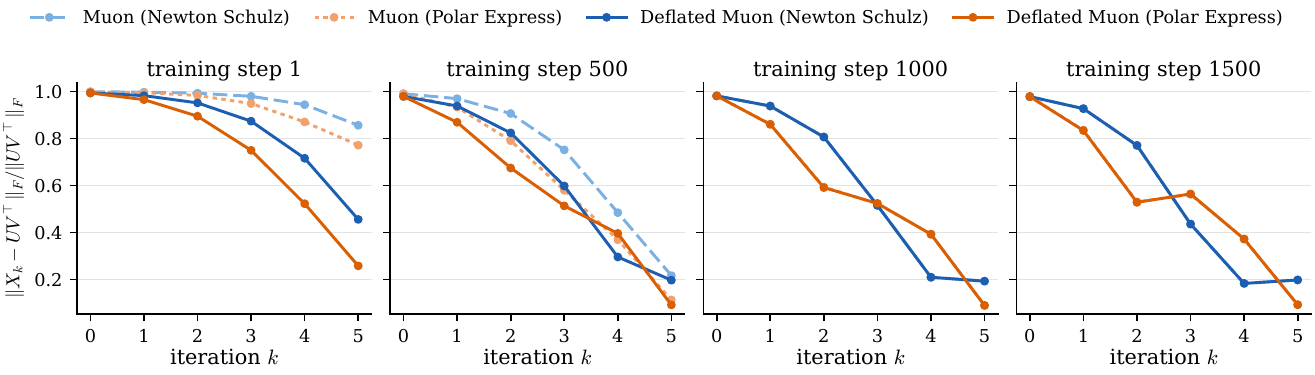}
\caption{MLP output projection.}
\label{fig:convergence_mlp_proj}
\end{subfigure}
\caption{Matrix-level accuracy on the two MLP momentum matrices of the middle
block of GPT-2 Large, recorded at training steps $1$, $500$, $1000$, and
$1500$, as a function of the number of iterations. Each panel reports the
relative error $\norm{X_L-UV^{\top}}_F/\norm{UV^{\top}}_F$ of the
Newton--Schulz and Polar Express mappings and of their deflated counterparts.
Where the gate does not fire (output projection at steps $1000$ and $1500$),
the deflated curve coincides with its baseline.}
\label{fig:convergence_mlp_all}
\end{figure}

\begin{figure}[p]
\centering
\begin{subfigure}[t]{\linewidth}
\centering
\includegraphics[width=0.9\linewidth]{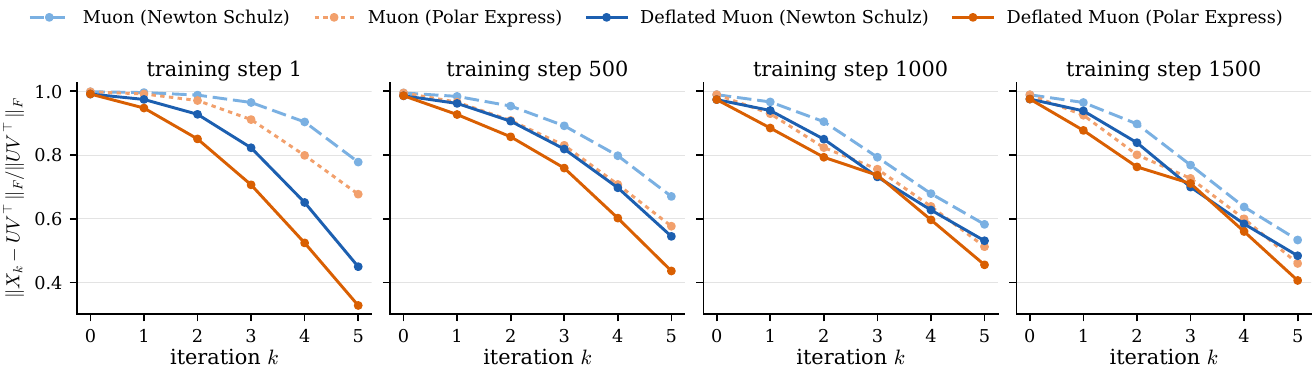}
\caption{Attention $\mathrm{Q}$ projection.}
\label{fig:convergence_attn_q}
\end{subfigure}\\[1.5ex]
\begin{subfigure}[t]{\linewidth}
\centering
\includegraphics[width=0.9\linewidth]{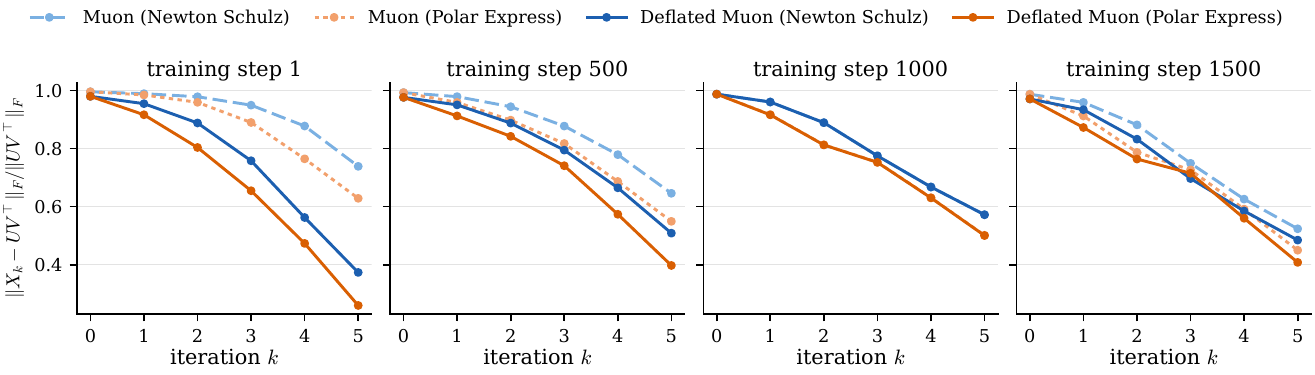}
\caption{Attention $\mathrm{K}$ projection.}
\label{fig:convergence_attn_k}
\end{subfigure}\\[1.5ex]
\begin{subfigure}[t]{\linewidth}
\centering
\includegraphics[width=0.9\linewidth]{gpt-large_convergence_attn_v.pdf}
\caption{Attention $\mathrm{V}$ projection (Figure~\ref{fig:convergence_attn_v} of the main text).}
\label{fig:convergence_attn_v_app}
\end{subfigure}\\[1.5ex]
\begin{subfigure}[t]{\linewidth}
\centering
\includegraphics[width=0.9\linewidth]{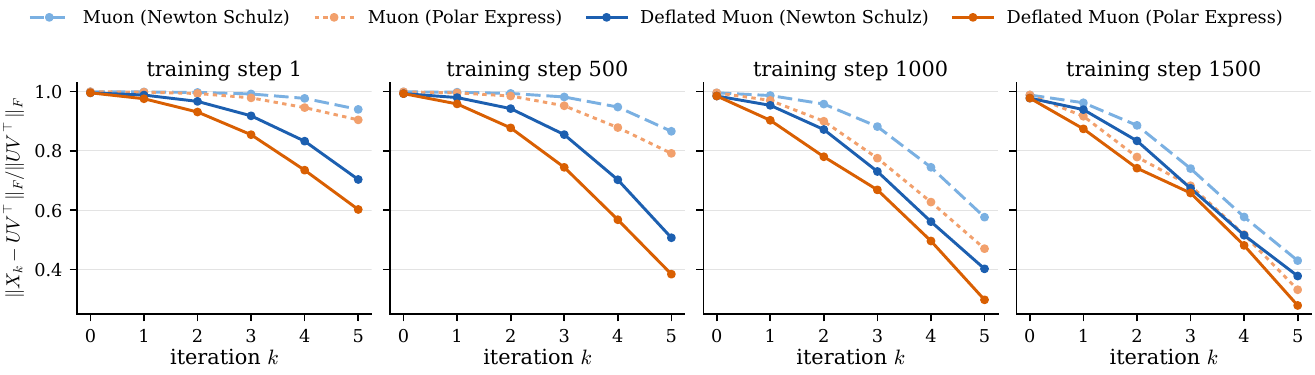}
\caption{Attention $\mathrm{O}$ projection.}
\label{fig:convergence_attn_o}
\end{subfigure}
\caption{Matrix-level accuracy on the four attention momentum matrices of the
middle block of GPT-2 Large, in the setting of Figure~\ref{fig:convergence_mlp_all}.
Where the gate does not fire ($\mathrm{K}$ at step $1000$), the deflated
curve coincides with its baseline; on $\mathrm{K}$ at step $1500$ the
estimated spectrum triggers the gate although the exact one sits just above
the threshold (see text).}
\label{fig:convergence_attn_all}
\end{figure}

\end{document}